\documentclass[11pt]{amsart}
\usepackage{graphicx}
\usepackage{fullpage}
\usepackage[boxed, ruled, vlined, linesnumbered]{algorithm2e}
\SetKwComment{Comment}{/* }{ */}

\usepackage{float}
\usepackage[dvipsnames]{xcolor}
\usepackage[toc,page,title,titletoc,header]{appendix}

\usepackage{listings}

\definecolor{mauve}{rgb}{0.58,0,0.82}

\usepackage{tikz}

\usepackage{graphicx}
\usepackage{capt-of}
\usepackage{appendix}
\usepackage{amsmath}
\usepackage{amssymb} 
\usepackage{amscd}
\usepackage{epsfig}
\usepackage{amsfonts}

\usepackage{hyperref}
\usepackage{cleveref}

\theoremstyle{plain}
\newtheorem{lem}{Lemma}
\newtheorem{theorem}{Theorem}

\theoremstyle{definition}

\theoremstyle{remark}

\newtheorem{conjecture}{Conjecture}

\newcommand{\tL}{\mathbf 1} 
\newcommand{\tO}{\mathbf 0}

\addtocontents{toc}{\protect\setcounter{tocdepth}{1}}

\begin{document}

\title[On the binary digits of $n$ and $n^2$]{On the binary digits of $n$ and $n^2$}

\author{Karam Aloui}
\address{1. Université de Tunis El Manar, Institut Supérieur des Technologies Médicales de Tunis, 9 Rue Zouhair Essafi, 1006, Tunis, Tunisia;
2. Université de Sfax, Laboratoire d’Algèbre, Géométrie et Théorie Spectrale, Route de la Soukra, km 3.5, 3000, Sfax, Tunisia}
\email{alouikaram@yahoo.fr}

\author{Damien Jamet}
\address{LORIA, Campus Scientifique BP 239, F-54506 Vand\oe uvre-l\`es-Nancy, France;}
\email{damien.jamet@loria.fr}

\author{Hajime Kaneko}
\address{Institute of Mathematics, University of Tsukuba, 1-1-1, Tennodai, Tsukuba, Ibaraki, 305-8571, JAPAN; 
Research Core for Mathematical Sciences, University of Tsukuba, 1-1-1, Tennodai, Tsukuba, Ibaraki, 305-8571, JAPAN}
\email{kanekoha@math.tsukuba.ac.jp}

\author{Steffen Kopecki}
\email{steffen.kopecki@gmail.com}

\author{Pierre Popoli}
\address{1. Universit\'e de Lorraine, Institut Elie Cartan de Lorraine, UMR 7502, Vandoeuvre-l\`es-Nancy, F-54506, France;
2. CNRS, Institut Elie Cartan de Lorraine, UMR 7502, Vandoeuvre-l\`es-Nancy, F-54506, France}
\email{pierre.popoli@univ-lorraine.fr}

\author{Thomas Stoll}
\address{1. Universit\'e de Lorraine, Institut Elie Cartan de Lorraine, UMR 7502, Vandoeuvre-l\`es-Nancy, F-54506, France;
2. CNRS, Institut Elie Cartan de Lorraine, UMR 7502, Vandoeuvre-l\`es-Nancy, F-54506, France}
\email{thomas.stoll@univ-lorraine.fr}

\begin{abstract}
Let $s(n)$ denote the sum of digits in the binary expansion of the integer $n$. Hare, Laishram and Stoll (2011) studied the number of odd integers such that $s(n)=s(n^2)=k$, for a given integer $k\geq 1$. The remaining cases that could not be treated by theses authors were $k\in\{9,10,11,14,15\}$. In this paper we show that there is only a finite number of solutions for $k\in\{9,10,11\}$ and comment on the difficulties to settle the two remaining cases $k\in\{14,15\}$. A related problem is to study the solutions of $s(n^2)=4$ for odd integers. Bennett, Bugeaud and Mignotte (2012) proved that there are only finitely many solutions and conjectured that $n=13,15,47,111$ are the only solutions. In this paper, we give an algorithm to find all solutions with fixed sum of digits value, supporting this conjecture, as well as show  related results for $s(n^2)=5$.
\end{abstract}

\maketitle


\section{Introduction}\label{secintro}

Let $s(n)$ be the sum of digits in the binary expansion (i.e. the Hamming weight) of $n \in \mathbb{N}$. In the present paper we investigate the question of whether or not the equation 
\begin{equation}\label{problem}
s(n)=s(n^2)=k
\end{equation} 
has infinitely many odd solutions in $n$ for a given $k \in \mathbb{N}$.\footnote{Note that for all integers $n$ we have $s(2n)=s(n)$. This means that the restriction to $n$ odd is necessary to make this question meaningful.} Hare, Laishram and Stoll~\cite{harelaishramstoll2011} settled all cases with the exception of $k\in\{9,10,11,14,15\}$. Our  contribution here is to show, via a combinatorial and algorithmic approach, that the equation only has finitely many solutions for $k\in\{9, 10, 11\}$. We will address the computational issues that we encounter for the last remaining open cases, namely, $k=14, 15$. 

The main motivation to consider~(\ref{problem}) comes from work of Madritsch and Stoll~\cite{MaSt14} who showed that $(s(n^2)/s(n))_{n\geq 1}$ is dense in $\mathbb{R}^+$.
This elaborates on an old result of Stolarsky~\cite{St78} (see also~\cite{HLS11-2, Lind97, Melfi05, Mei15, Saund15}) who showed that
$\liminf_{n\to \infty} {s(n^2)}/{s(n)}=0.$ Since the average size of $s(n^2)$ is twice as large as that of $s(n)$ (see~\cite{BK95, Pe02}) the equation~(\ref{problem}) concerns an exceptional set of integers. In particular, it is intriguing that for certain values of $k$ the equation allows for infinitely many odd solutions $n$ and for other values of $k$ there is just a finite number. One of the results of Hare, Laishram and Stoll~\cite{harelaishramstoll2011} states that there are infinite parametric families of solutions for $k=12, 13$ and $k\geq 16$. They showed that 
\begin{align}
&s(n)=s(n^2)=12,\mbox{ for all } n=111\cdot 2^t+111, \mbox{ with }t\geq 15, \label{s(n)=s(n2)=12} \\
&s(n)=s(n^2)=13,\mbox{ for all } n=23\cdot 2^t+1471, \mbox{ with }t\geq 21, \label{s(n)=s(n2)=13} \\
& s(n)=s(n^2)=16,\mbox{ for all }n=111\cdot 2^{t}+1919, \mbox{ with } t\geq 21. \label{s(n)=s(n2)=16}
\end{align}

On the other side of the spectrum, there are only finitely many solutions for $k\leq 8$. For example, for $$s(n)=s(n^2)=8,$$ there are only 64 solutions in odd integers and the largest solution is $n=266335$ (see~\cite[Table~2]{harelaishramstoll2011}). These results are based on an algorithm that handles all the possible orderings of the exponents in $n^2$ when $n$ is written as a sum of a small number of powers of 2. Since the algorithm treats (in an exhaustive way) all cases, the method of~\cite{harelaishramstoll2011} allowed to explicitly determine all the solutions for $k\leq 8$. The running time of the algorithm, however, explodes for larger values of $k$. Some heuristic arguments are given in~\cite[Section~5]{harelaishramstoll2011} to support the conjecture that there are only finitely many solutions for $k\in\{9,10\}$. The main purpose of this paper is to combine a new combinatorial factorization lemma and two algorithms with recent results by Kaneko and Stoll~\cite{kanekostoll2021} to reduce the investigation to a finite case analysis. We then carry out this case analysis for $k\in\{9,10,11\}$ in a unified manner.

\section{Main results}

In the present article we show the following theorem: 

\begin{theorem}\label{main_theorem}
Let $k\in\{9,10,11\}$. Then the number of odd integers $n$ with $$s(n)=s(n^2)=k$$ is finite.
\end{theorem}

We have made a global search for $s(n^2)=s(n)=k$, $11\leq k \leq 15$, up to $n<2^{80}$ (see \Cref{sec-rem1415}). No infinite family occurs clearly in the cases $k=14$ and $k=15$ compared to $k\in \{12,13,16\}$, see~\eqref{s(n)=s(n2)=12}--\eqref{s(n)=s(n2)=16}. We therefore formulate the following conjecture:

\begin{conjecture} \label{conj 14-15}
Let $k \in \{14,15\}$. Then the number of odd integers $n$ with $s(n)=s(n^2)=k$ is finite. 
\end{conjecture}

For Theorem~\ref{main_theorem} it is crucial to have efficient algorithms at our disposal that calculate certain sets that appear in intermediate steps in the proof.

\medskip

For  fixed $\ell_1 \geq 1$, $\ell_2 \geq 1$ and $m\geq 1$, set
\begin{equation}\label{deltadef}
\Delta_{\ell_1,\ell_2,m}:=\{n\in \mathbb{N}: \quad s(n)=\ell_1, \quad s(n^2)\leq \ell_2, \quad n<2^m, \quad n\text{ odd} \}. 
\end{equation}

Our first algorithm, called \texttt{next}, has the purpose to calculate $\Delta_{\ell_1,\ell_2,m}$ for small values of $\ell_1,\ell_2,m$. 
\medskip

For fixed $k\in\mathbb{N}$ and $\lambda\geq 1$, let
\begin{equation}\label{defElambdak}
E_{k,\lambda}:=\{n\in \mathbb{N}: \quad s(n^2)=k, \quad s(n)=\lambda, \quad n\mbox{ odd} \},
\end{equation}
and set
\begin{equation}\label{defEk}
  E_k := \bigcup_{\lambda\geq 1} E_{k,\lambda}.
\end{equation}

The aim of the second algorithm, called \texttt{max-integer}, is to calculate efficiently $E_{k, \lambda}$ for small $k$ and $\lambda$. Several results about the sets $E_k$ are already known in the literature. To begin with, it is an elementary calculation to show that $E_2=\{3\}$. Szalay~\cite{szalay2002} showed that 
\begin{equation}\label{s(n2)=3}
E_3 =\{2^{t}+1: \; t\geq 2\} \cup \{7, 23\},
\end{equation}
see also~\cite{Lu04} for a generalization. Szalay's proof relies on a result of Beukers on the Ramanujan--Nagell equation.
Typically such sets are composed of a union of a set of infinite parametrized integers and a set of small sporadic solutions. Bennett/Bugeaud/Mignotte~\cite{bennettbugeaudmignotte2012}, Bennett/Bugeaud~\cite{BeBu14}, Hajdu/Pink~\cite{HaPi14} and B\'erczes/Hajdu/Miyazaki/Pink~\cite{BHMP16} generalized Szalay's results to other bases and more general powers, see also Bennett~\cite{Be17}. Finally, we mention also the recent work of Szalay~\cite{Sz20} who considered algorithms to find the solution set of the Diophantine equation $2^n+\alpha\cdot 2^m+\alpha^2=x^2$, where $\alpha$ is a fixed positive integer.

\medskip

As for $E_4$, it is known that it is a finite set (so, no infinite families occur), see Bennett, Bugeaud and Mignotte~\cite{bennettbugeaudmignotte2012}, and Corvaja and Zannier~\cite{corvaja}. Bennett, Bugeaud and Mignotte conjectured:
\begin{conjecture}\label{conjectureE4}
\begin{equation}\label{conjE4}
  E_4=\{13,\; 15,\; 47,\; 111\}.  
\end{equation}
\end{conjecture}
This conjecture remains still open. 

\medskip

We here consider a refined version of Conjecture~\ref{conjectureE4}, namely, we restrict our attention to those integers $n$ that have a fixed sum of binary digits. This is particularly valuable in the study of~\eqref{problem}, where we need to know explicitly $E_{4,\lambda}$ and $E_{5,\lambda}$ for small values of $\lambda$. Indeed, the infinite family given in~\eqref{s(n)=s(n2)=13} is built from the two integers $23$ and $1471$. Since $s(23^2)=3$, $s(1471^2)=5$ and $s(23\cdot1471)=5$, we have the correct amount of bits in the square of $n=23 \cdot 2^{t} +1471$ for sufficiently large $t$. 

\medskip
We apply \texttt{max-integer} to show the following result.

\begin{theorem}\label{thm_s(n2)=4 and 5}
We have
\begin{equation}\label{union4lambda}
\bigcup_{\lambda\leq 17 }E_{4,\lambda}=\{13,\;15,\;47,\;111\}    
\end{equation}
and
\begin{equation}\label{union5lambda}
\bigcup_{4\leq \lambda\leq 15 }E_{5,\lambda}=\{29,\;31,\;51,\;79,\;91,\;95,\;157,\;223,\;279,\;479,\;727,\;1471,\;5793\}.
\end{equation}
Moreover,
\begin{equation}\label{Edecomp}
  E_5=\{1+2+2^\ell:\; \ell \geq 3\} \; \cup  \;
\{1+2^{\ell}+2^{\ell+1}:\; \ell \geq 3\} \; \cup  \;
\{1+2^{\ell}+2^{2\ell-1}:\; \ell \geq 3\}\; \cup \; E',  
\end{equation}
where $E'$ is a finite set.
\end{theorem}


This theorem gives more evidence on~(\ref{conjE4}): all $E_{4,\lambda}$ are empty sets for $5\leq \lambda\leq 17$. The result might also point towards a possible computational proof if we could establish a universal bound for $\lambda$.
For~(\ref{Edecomp}) we conjecture:
\begin{conjecture}
$$E'=\{29,\;31,\;51,\;79,\;91,\;95,\;157,\;223,\;279,\;479,\;727,\;1471,\;5793\}.$$
\end{conjecture}
This is supported by the fact that the largest weight of an integer in $E'$ occurs for $n=1471$, namely $s(1471)=9$, so that the sets $E_{5,\lambda}$ are all empty for $10\leq \lambda\leq 15$.


\bigskip
\bigskip

The paper is structured as follows. In Section~\ref{sec-aux}, we state, collect and prove some related 
auxiliary results; in particular, we give a new combinatorial factorization lemma that is at the core of our method. Section~\ref{sec-pr} is devoted to the proof of \Cref{main_theorem}, where we make use of the algorithm \texttt{next}. In Section~\ref{sec-45}, we give the proof of \Cref{thm_s(n2)=4 and 5} where we rely on the algorithm \texttt{max-integer}. We postpone, for an easier readability, the detailed description of the two algorithms and their implementation to Section~\ref{sec-algos}. We finally conclude in Section~\ref{sec-rem1415} with some remarks on the remaining cases of~\eqref{problem}, i.e. the cases $k=14$ and $k=15$, that remain unsettled.

\section{Preliminaries}\label{sec-aux}

Let $n=\sum_{i=0}^{\ell} \varepsilon_i 2^i$ with $\varepsilon_i\in\{0,1\}$ and $\varepsilon_\ell=1$. We write $(n)_2$ to refer to the binary expansion of the integer $n$. Recall that $s(n)=\sum_{i=0}^{\ell} \varepsilon_i$ is the sum of digits of $n$. We use the letter $\mathsf x$ (with or without indices) to denote binary blocks that always end in $\tL$, so that they correspond to the binary expansions of odd integers. In the language of combinatorics on words, $\mathsf x$ will be a non-empty word over the alphabet $\{\tO, \tL\}$ (the ``bits'') whose rightmost symbol is $\tL$. We always use boldface to talk about bits. To keep notation as simple and readable as possible, we will use $x$ (note the change in the font) to denote the associated odd integer. For example, for  $\mathsf x=\tL \tL$ we have $x=3$; also, we write $1$ for the integer ``one'', and $\tL$ for the one-bit. Note that $\ell=\ell((n)_2)$ is the length of the binary expansion of the integer $n$. Again, for simplicity reasons, we also use $\ell(n)$ for $\ell((n)_2)$. As usual, in combinatorics on words, we will write $\mathsf{x}\mathsf{y}$ for the concatenation of the binary words $\mathsf{x}$ and $\mathsf{y} $, and $\mathsf{x}^\ell$ for the $\ell$-fold concatenation of the word $\mathsf{x}$. Moreover, we write $|\mathsf{x}|$ for the length of the word $\mathsf{x}$. On the other hand, $xy$, or $x\cdot y$, will systematically denote the multiplication of the integers $x$ and $y$.

For an odd integer $n$ such that $s(n)=k$, we consider a decomposition of its binary expansion in the form
\begin{equation}\label{decomp}
(n)_2=\mathsf{x}_m \tO^{\ell_m} \mathsf{x}_{m-1} \cdots \mathsf{x}_1 \tO^{\ell_1} \mathsf{x}_0
\end{equation}
into $m<k$ blocks of $\tO$-bits of length $\ell_i$, for $1 \leq i \leq m$, which separate $\mathsf{x}_i$, for $0 \leq i \leq m$. Recall that all $\mathsf{x}_i$ end in $\tL$-bits, so that they correspond to the binary expansion of odd integers. Note that the decomposition~(\ref{decomp}) is not unique since we can merge or split inner blocks to obtain other factorizations. Let us denote $y_{i,j}=x_i\cdot x_j$ for $0 \leq i,j \leq m$, and $\mathsf{y}_{i,j}$ the associated binary blocks that will compose $n^2$. Note that $\mathsf{y}_{i,j}$ again ends in a $\tL$-bit. Let $$\hat{\ell_j}=\sum \limits_{i=1}^j (\ell_i +\ell(x_{i-1}))$$ for $0\leq j \leq m$, which represents the length of $\tO^{\ell_j} \mathsf{x}_{j-1} \cdots \mathsf{x}_1 \tO^{\ell_1} \mathsf{x}_0$. Furthermore let $\ell_{i,i}=2\hat{\ell_i}$ for $0\leq i\leq m$ and $\ell_{i,j}=\hat{\ell_i}+\hat{\ell_j}+1$ for $0\leq i,j \leq m$ and $i\neq j$. Then the square $n^2$ is the sum of the integers \begin{align*}
&u_{i,i}=2^{\ell{i,i}}y_{i,i} \quad \text{for } 0\leq i \leq m \text{ and }\\
&u_{i,j}=2^{\ell{i,j}}y_{i,j} \quad \text{for } 0\leq i <j \leq m.
\end{align*}

We say that the summand $y_{i,j}$ \textit{interferes} with the summand $y_{i',j'}$ if, in the addition of the two terms written in binary, a carry propagation caused by $y_{i,j}$ reaches a binary bit of $y_{i',j'}$, or vice-versa (we take liberty to say also, that $\mathsf{y}_{i,j}$ interferes with $\mathsf{y}_{i',j'}$). We will frequently discuss the situation on how many $\tL$-bits remain in the addition of interfering terms. We will reject possibilities when the additions lead to numbers with too many $\tL$-bits. If blocks are \textit{non-interfering} then the number of $\tL$-bits of the their sum is the sum of the $\tL$-bits of the summands. Let us explain the procedure with an example. If we add $\tL\tL\tL$ or its shifts to $\tL \tL \tL \tO \tO \tO\tO\tL$, we observe that it is impossible to find shifts in a way that the sum of the two summands gives a single $\tL$-bit:
\begin{align*}
  &\quad \tL \tL \tL \tO \tO \tO\tO\tL&\\
   +\quad&\longleftarrow\tL\tL\tL \longrightarrow
\end{align*}

This procedure can be easily implemented: it is sufficient to test via one \texttt{for}-loop. If more than two terms are added together, then more \texttt{for}-loops will do the job.

\medskip
 The next lemma gives a sufficient condition for non-interference between two summands. 

\begin{lem}\label{interference}
Let  $y_{i,j}$, $y_{i',j'}$ be two summands defined as before. If $\ell_{i,j} \geq \ell_{i',j'} + \ell(y_{i',j'}) + k^2$, then $y_{i,j}$ does not interfere with $y_{i',j'}$. \end{lem}

\begin{proof}
If these two summands interfered, then there would be at least $k^2$ $\tL$-bits involved in the carry propagation from $y_{i',j'}$ to $y_{i,j}$. But the number of $\tL$-bits in all summands is at most $k+\frac{k(k-1)}{2}<k^2$. 
\end{proof}

Let $\ell_{\min}=\min_{1\leq i \leq m} (\ell_i)$ and $\ell_{\max}=\max_{0\leq i \leq m} \ell(x_i) $. By \Cref{interference} we can deduce that if $\ell_{\min}>2\ell_{\max}+k^2$, then two summands $y_{i,j}$, $y_{i',j'}$ do not interfere if $i>i'$ and $j\geq j'$.

\begin{lem}[Factorization lemma]\label{factorization}
For $k\geq 1$ there is a bound $N_k$ such that the binary expansion of every odd $n\geq N_k$ that satisfies $s(n)=s(n^2)=k$ can be factorized as
\begin{equation}\label{decompfactorization}
(n)_2=\mathsf{x}_m \tO^{\ell_m} \mathsf{x}_{m-1} \cdots \mathsf{x}_1 \tO^{\ell_1} \mathsf{x}_0
\end{equation} 
where $1\leq m <k$, $\mathsf{x}_0,\ldots, \mathsf{x}_m$ are the binary words corresponding to the binary expansions of odd integers and $\ell_1,\ldots,\ell_m \in \mathbb{N}$ such that $\ell_{\min}>2x_{\max} +k^2$ where $\ell_{\min}=\min_{1\leq i \leq m}(\ell_i)$ and $x_{\max}=\max_{0\leq i \leq m}  |\mathsf{x}_i|  $.
\end{lem}

\begin{proof}
Let $f(i)=4i+k^2$, let $N_k=2^{f^k(1)}$ where $f^k(1)$ is the $k$-fold composition of $f$ evaluated at $1$. Consider an odd integer $n\geq N_k$ such that $s(n)=k$. If the binary expansion of $n$ contains $k-1$ blocks of $\tO$-bits and if each of these $\tO$-blocks is longer than $k^2+2$, then each $\tL$-bit in the expansion forms one of the $x_i$ with $m=k-1$ and we are done. Otherwise, we combine all $\tO$-blocks which have a length at most $k^2+2$ with its bordering $\tL$-bits and make this one of the factors $x_i$. If all remaining $\tO$-blocks are longer than $2(k^2+4)+k^2$, we find a suitable factorization of $(n)_2$ with $m=k-2$. Otherwise, we continue inductively and obtain $n>2^{f^k(1)}$ has a desired factorization. 
\end{proof}

\bigskip

We will also need some elementary results on multiples of $3$ with few non-zero digits.

\begin{lem} \label{useful_1}
Let $n$ be an odd integer with $s(3n)=2$. Then $(n)_2\in \{  (\tL \tO)^\ell \tL \tL: \; \ell\geq 0\}\cup \{\tL\}$.
\end{lem}

\begin{proof}
 For $n\geq 5$ we set $(n)_2=\tL \varepsilon_d \varepsilon_{d-1}\cdots \varepsilon_0 \tL$ $(\varepsilon_i\in\{\tO, \tL\}, d\geq 0)$ and observe that the usual addition $2n+n$ translates into
\begin{align*}
  &\tL \;\;\varepsilon_d \;\;\varepsilon_{d-1}\;\;\cdots \;\;\varepsilon_1\;\;\varepsilon_0 \;\;\;\tL&\\
   +\quad&\;\;\;\;\;\tL\;\;\; \varepsilon_d\;\;\; \varepsilon_{d-1}\;\cdots \;\;\varepsilon_1\;\;\varepsilon_0\;\; \tL.
\end{align*}
Since the last $\tL$-bit will stay after the addition, the addition of the penultimate $\tL$ to $\varepsilon_0$ must give rise to a carry that propagates up to the highest significant digits. The only way to make this happen without creating additional $\tL$-bits in the sum is $\varepsilon_0=\varepsilon_2=\varepsilon_4=\cdots=\tL$ and $\varepsilon_1=\varepsilon_3=\varepsilon_5=\cdots=\tO$.
\end{proof}

\begin{lem}\label{s(3x)=4}
Let $n$ be an odd integer with $s(3n)=4$. Then $(n)_2=\mathsf{x}_1 \tO^s \mathsf{x}_0$ for some $s \geq 2$ with 
$$\mathsf{x}_1,\mathsf{x}_0\in \{(\tL \tO)^\ell \tL \tL: \; \ell\geq 1\}\cup \{\tL\},$$
or $n\leq 2^{2s(n)-1}$.
\end{lem}

\begin{proof}
If there is a block of $\tO$'s of length $\geq 2$ inside $(n)_2$, then in the addition of $2n+n=3n$ there are non-interfering terms and the additions have to amount for $2$ bits in the sum of the corresponding portions. Lemma~\ref{useful_1} shows that the only possibilities are the blocks $(\tL \tO)^\ell \tL \tL$ for some $\ell \geq 0$, and the block consisting of a single $\tL$. This gives the first part in the statement. If there is no block of consecutive $\tO$'s of length $\geq 2$ then $x$ is evidently bounded by $2^{2s(n)}-1$.
\end{proof}


\begin{lem}\label{s(3x)=3}
Let $n$ be an odd integer with $s(3n)=3$. Then
$$(n)_2\in \{\tL(\tO \tL)^{\ell_1}(\tL \tO)^{\ell_2} \tL \tL: \; \ell_1,\ell_2\geq 0\}$$
\end{lem}

\begin{proof}
Recall the reasoning of the proof of Lemma~\ref{useful_1} first. A possible carry has to propagate all the way up to the highest significant digits in order to generate only $\tO$-bits except the lowest significant bit and the highest significant bit. In the former proof, this implied an alternation of $\tO$-bits and $\tL$-bits in the middle part. In the statement of the present lemma, since we want three $\tL$-bits in the resulting sum, we need to break this alternation at least once. We therefore have the following addition: 

\begin{align*}
  &\tL \;\;\varepsilon_d \; \; \cdots \;\;\varepsilon_{k} \;\; \tL \;\;\; \underline{\tL}\;\; \tO \;\;\cdots \;\; \tL \;\;\tO\;\;\tL \;\;\tL&\\
   +\quad&\;\;\;\;\tL\;\;\; \varepsilon_d\;\;\; \cdots \;\; \varepsilon_{k}\; \underline{\tL} \;\; \tL \;\; \tO\; \; \cdots \;\; \tL \;\;\tO\;\;\tL\;\; \tL.
\end{align*}

The lowest significant $\tL$-bit will stay after the summation of the two numbers (it does not interact with the other bits). In the overlapping $\tL$-bits at the breaking point (underlined in the above addition scheme) there will remain one $\tL$-bit in the sum. The addition of these bits generates a carry that has to generate only $\tO$-bits up to the highest significant $\tL$-bit. The only way to achieve this is again to alternate the $\tO$-bits and $\tL$-bits. This directly translates into the given form. 
\end{proof}


\begin{lem} \label{useful_2}
If $(n)_2\in\{(\tL\tO)^{\ell}\tL\tL:\; \ell\geq 2\}$ then $s(n^2)\geq 7$. 
\end{lem}
\begin{proof}
An elementary calculation shows that, if $(n)_2=(\tL\tO)^{\ell}\tL\tL$, $\ell\geq 1$, then
$$(n^2)_2= 
\begin{cases} 
{ (\tL\tL\tL\tO\tO\tO)^{j-1}\tL\tL\tL\tO\tO\tL(\tO\tO\tO\tL\tL\tL)^{j}\tO\tO\tL} &\mbox{if } \ell=3j \mbox{  for some  } j\geq 1,\\
{ (\tL\tL\tL\tO\tO\tO)^{j}\tL\tL\tL\tL(\tO\tO\tO\tL\tL\tL)^{j}\tO\tO\tL} &\mbox{if } \ell=3j+1 \mbox{  for some  } j\geq 0,\\
{  (\tL\tL\tL\tO\tO\tO)^{j}\tL\tL\tL\tO\tO\tL\tL\tL(\tO\tO\tO\tL\tL\tL)^{j}\tO\tO\tL} &\mbox{if } \ell=3j+2 \mbox{  for some  } j\geq 0.
\end{cases}
$$
\end{proof}

In our application for the infinite family in Lemma~\ref{s(3x)=4}, we will fix the value of $s(n)$, which means that $\ell$ is small. Lemma~\ref{useful_2} then guarantees that the squares of such integers have (too) many $\tL$-bits, which will lead to a contradiction. We will make use of this procedure at several places in our investigation, in particular to check that there are no solutions in odd integers $n$ for the system $s(3n)=4$, $s(n)=9$ and $s(n^2)=5$. The sporadic solutions that are bounded in Lemma~\ref{s(3x)=4} can be checked directly by an exhaustive computer search.

\medskip

We next recall two recent results by Kaneko and Stoll~\cite{kanekostoll2021} that deal with products of integers with few binary digits.  

\begin{lem}\label{s(ab)=2}
Let $\ell,m\geq 2$, and let $a$ and $b$ be two odd integers such that $s(a)=\ell$ and $s(b)=m$. If $s(ab)=2$, then we have \begin{align*}
ab<2^{2\ell m-4}.
\end{align*}
\end{lem}

\begin{lem}\label{s(ab)=3}
Let $\ell,m\geq 2$, and $a$ and $b$ be two odd integers such that $s(a)=\ell,s(b)=m$ and $m\ell \geq 5$. If $s(ab)=3$, then we have \begin{align*}
ab<2^{4\ell m-13}.
\end{align*} 
\end{lem}

We will use these results when we look for solutions in the form $\mathsf{x}_1\tO\cdots\tO \mathsf{x}_0$ for a large inner block of $\tO$-bits. For such a structure, we have three separated contributions to the binary decomposition in the square: $x_1^2$, $x_0^2$ and the double product $x_1\cdot x_0$ which do not interfere since they are well-separated. When $s(x_1x_0)=2$ or $s(x_1x_0)=3$, we can apply these two lemmas to bound $x_1$ and $x_0$, and an exhaustive search will then be sufficient to conclude. We mention that the direct analogue to Lemma~\ref{s(ab)=2} and Lemma~\ref{s(ab)=3} for $s(ab)=4$ does not hold true (see~\cite{kanekostoll2021}).

\section{Proof of \Cref{main_theorem}}\label{sec-pr}

According to \Cref{factorization}, if there exist infinitely many odd solutions of~\eqref{problem} for some $k$, then almost all (i.e. all with a finite number of exceptions) of the binary expansions of these solutions can be factorized. Consider a factorization of $(n)_2$ as stated in \Cref{factorization} and note that none of the $2m+1$ summands in the set $\{y_{m,i}\}_{i=0}^m\cup \{y_{i,0}\}_{i=0}^m$ interfere with each other. Some of these summands may interfere with other summands, yet, even in that cases, each contributes with at least one $\tL$-bit to the binary expansion of $n^2$. Thus, if $m\geq k/2$, then $s(n^2)\geq 2m+1 >k$ and $n$ cannot be a solution of~\eqref{problem}. We therefore can safely suppose that $1\leq m<k/2$. 

We have the corresponding graphs that show the various possibilities of interference. Herein, vertices are the summands and the edges correspond to possible instances of interference between summands. 

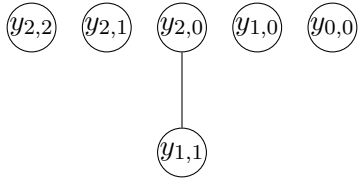
\begin{figure}[ht]
\begin{tikzpicture}
    \tikzstyle{every node}=[circle, draw, fill=white!50, inner sep=0pt, minimum width=4pt]
    \draw \foreach \x in{0,...,2}{
    (0-\x,1) node[above]{$y_{2,\x}$}
    };
    \draw \foreach \x in{0,...,1}{
    (2-\x,1) node[above]{$y_{\x,0}$}
    };
    \draw (0,0) node[below]{$y_{1,1}$};
    \draw (0,0) -- (0,1);
\end{tikzpicture}
\caption{Interference graph for $m=2$.}
\label{fig m=2}
\end{figure}

\begin{figure}[ht]
\begin{tikzpicture}
    \tikzstyle{every node}=[circle, draw, fill=white!50, inner sep=0pt, minimum width=4pt]
    \draw \foreach \x in{0,...,3}{
    (0-\x,1) node[above]{$y_{3,\x}$}
    };
    \draw \foreach \x in{0,...,2}{
    (3-\x,1) node[above]{$y_{\x,0}$}
    };
    \draw \foreach \x in{1,...,2} {
    (1-\x,0) node [below]{$y_{2,\x}$}
    };
    \draw (1,0) node [below]{$y_{1,1}$};
    \draw (0,0) -- (0,1);
    \draw (-1,0) -- (0,1);
    \draw (-1,0) -- (-1,1);
    \draw (1,0) -- (1,1);
    \draw (1,0) -- (0,1);
\end{tikzpicture}
\caption{Interference graph for $m=3$.}
\label{fig m=3}
\end{figure}
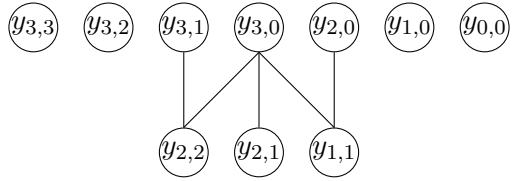

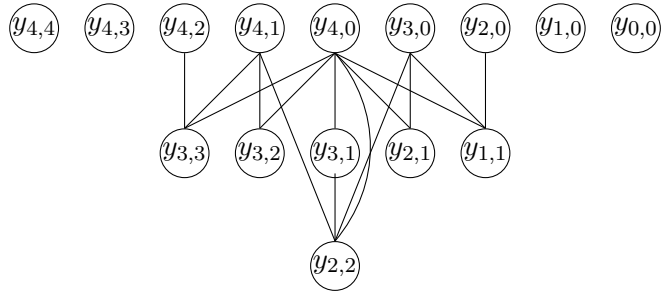
\begin{figure}[ht]
\begin{tikzpicture}
    \tikzstyle{every node}=[circle, draw, fill=white!50, inner sep=0pt, minimum width=4pt]
    \draw \foreach \x in{1,...,4}{
    (5-\x,2) node[above]{$y_{4,\x}$}
    };
    \draw \foreach \x in{0,...,4} {
    (9-\x,2) node [above]{$y_{\x,0}$}
    };
    \draw \foreach \x in{1,...,3}{
    (6-\x,1) node[below]{$y_{3,\x}$}
    };
    \draw \foreach \x in{1,...,2} {
    (8-\x,1) node [below]{$y_{\x,1}$}
    };
    \draw (5,-0.5) node[below]{$y_{2,2}$};
    \draw \foreach \x in{1,...,5}{
    (2+\x,1) --(2+\x,2)
    };
    \draw (3,1) -- (4,2);
    \draw (3,1) -- (5,2);
    \draw (4,1) -- (5,2);
    \draw (6,1) -- (5,2);
    \draw (7,1) -- (6,2);
    \draw (7,1) -- (5,2);
    \draw(5,-0.5) to[bend right=0] (6,2);
    \draw(5,-0.5) to[bend left=0] (4,2);
    \draw(5,-0.5) -- (5,0.4);
    \draw(5,-0.5) to[bend right=40] (5,2);
\end{tikzpicture}
\caption{Interference graph for $m=4$.}
\label{fig m=4}
\end{figure}

\newpage 

As a consequence of \Cref{factorization}, the following result implies \Cref{main_theorem}.

\begin{lem}\label{main_lemma} 
Let $k\in\{9,10,11\}$ and $n\geq N_k$ be a sufficiently large odd integer with $$s(n)=s(n^2)=k.$$ Then there exists no factorization of $(n)_2$ in the form ~(\ref{decompfactorization}).
\end{lem}

For $\boxed{k=9}$ and $\boxed{k=10}$ we have $1\leq m \leq 4$, and for $\boxed{k=11}$ we have $1\leq m \leq 5$. The proof of \Cref{main_lemma} is based on elementary considerations but we need to combine several ingredients to conclude: Szalay's result on $E_3$, Kaneko and Stoll's results on products of integers with $2$ or $3$ digits, non-existence of certain squares modulo powers of $2$, analysis of possible interference for multiple blocks, the tables of Hare, Laishram and Stoll~\cite{harelaishramstoll2011} etc. The investigation results in a finite case analysis where the details depend on the parameters. For the convenience of the reader, we have arranged the proof for fixed $m$ since the reasoning is very similar for $k\in\{9,10,11\}$ when the interference graph stays the same. We use freely our algorithms \texttt{next} and \texttt{max-integer}, whose descriptions are postponed to Section~\ref{sec-algos}. 

\subsection{The case $m=1$}\label{subsecm1}

We have $(n)_2=\mathsf{x}_1 \tO \cdots \tO \mathsf{x}_0$ with a (long) contiguous inner block of $\tO$-bits.  First, we observe that if $x_0=1$ (recall that this is the same as saying that $\mathsf{x}_0=\tL$), then $n$ cannot be a solution of~\eqref{problem} since then we would have $s(n^2)=s(x_1^2)+s(x_1)+1=s(x_1^2)+(k-1)+1 > k$. Similarly, we have $x_1 \neq 1$. By symmetry we can suppose that $s(x_1)\geq s(x_0)$. We therefore have to solve the following system:
\bigskip

\begin{equation}\label{systeme}
\left\{
    \begin{array}{lll}
        s(x_1)+s(x_0)=k,\\
        s(x_1^2)+s(x_1x_0)+s(x_0^2)=k, \\
        s(x_1),\; s(x_0),\; s(x_1^2),\; s(x_0^2),\; s(x_1\cdot x_0)\geq 2,\\
        s(x_1)\geq s(x_0).
    \end{array}
\right.
\end{equation}
\bigskip

We distinguish the cases according to the value of $s(x_1\cdot x_0) \in \{2,\ldots,k-4\}$. 
\medskip
\begin{enumerate}
\item \underline{$s(x_1\cdot x_0)=2$.} Here, \Cref{s(ab)=2} provides an upper bound for $x_1$ and $x_0$. Recall the definition of $\Delta_{\ell_1,\ell_2,m}$ given in~(\ref{deltadef}). The following table lists all sets to check for possible solutions $(x_1,x_0)$ of~(\ref{systeme}):
\bigskip

\begin{center}
\begin{tabular}{|c|c|c|c|c|} \hline 
Sets $\Delta$ for $k=9$ &  Sets $\Delta$ for $k=10$ & Sets $\Delta$ for $k=11$  \\ \hline\hline
$\Delta_{5,5,36} \times \Delta_{4,5,36}$ & $\Delta_{5,6,46} \times \Delta_{5,6,46}$ & $\Delta_{6,7,56} \times \Delta_{5,7,56}$ \\ \hline 
$\Delta_{6,5,36} \times \Delta_{3,5,36}$ & $\Delta_{6,6,46} \times \Delta_{4,6,46}$ & $\Delta_{7,7,56} \times \Delta_{4,7,56}$\\ \hline 
$\Delta_{7,5,36} \times \Delta_{2,5,36}$ & $\Delta_{7,6,46} \times \Delta_{3,6,46}$ & $\Delta_{8,7,56} \times \Delta_{3,7,56}$  \\ \hline 
- & $\Delta_{8,6,46} \times \Delta_{2,6,46}$ &  $\Delta_{9,7,56} \times \Delta_{2,7,56}$ \\ \hline
\end{tabular}
\end{center}

\bigskip

We construct these sets with our algorithm \texttt{next} in an efficient manner and check whether it gives a solution to the system~(\ref{systeme}); there is no solution. We could also use the tables of~\cite{harelaishramstoll2011} for the cases $\boxed{k=9}$ and $\boxed{k=10}$, however, for the case $\boxed{k=11}$, the tables of~\cite{harelaishramstoll2011} do not allow to conclude and we need a new method to construct the related sets.
\medskip

\item \underline{$s(x_1\cdot x_0)=3$.} With the help of \Cref{s(ab)=3} we can, similarly to before, restrict our attention to a finite number of sets. These sets are given in the following table: 

\bigskip

\begin{center}
\begin{tabular}{|c|c|c|c|c|} \hline 
Sets $\Delta$ for $k=9$ &  Sets $\Delta$ for $k=10$ & Sets $\Delta$ for $k=11$  \\ \hline\hline
$\Delta_{5,4,67} \times \Delta_{4,4,67}$ & $\Delta_{5,5,87} \times \Delta_{5,5,87}$ & $\Delta_{6,6,107} \times \Delta_{5,6,107}$\\ \hline 
$\Delta_{6,4,67} \times \Delta_{3,4,67}$ & $\Delta_{6,5,87} \times \Delta_{4,5,87}$ & $\Delta_{7,6,99} \times \Delta_{4,6,99}$\\ \hline 
$\Delta_{7,4,67} \times \Delta_{2,4,67}$ & $\Delta_{7,5,87} \times \Delta_{3,5,87}$ & $\Delta_{8,6,83} \times \Delta_{3,6,83}$ \\ \hline 
- & $\Delta_{8,5,87} \times \Delta_{2,5,87}$&  $\Delta_{9,6,59} \times \Delta_{2,6,59}$ \\ \hline
\end{tabular}
\end{center}

\bigskip

In the last column we have already reduced the number of cases to consider in order to speed up the calculations. Again, as before, there is no solution.

\medskip
 
\item \underline{$s(x_1\cdot x_0)=4$.} 
We here investigate the weight of the square parts since there is no universal bound on $x_1\cdot x_0$.
\begin{itemize}
    \item Let $\boxed{k=9}$. Then we have $\max(s(x_1^2),s(x_0^2))=3$ and $\min(s(x_1^2),s(x_0^2))=2$. Recall the definitions of $E_{k,\lambda}$ and $E_k$ in~(\ref{defElambdak}) and~(\ref{defEk}), and Szalay's result~(\ref{s(n2)=3}). We have $x_1\in E_3$ and $x_0=3$, or $x_0\in E_3$ and $x_1=3$. Then $s(x_1)+s(x_0)\leq 6<9$, which is a contradiction.
    \item Let $\boxed{k=10}$. \begin{enumerate}
\item If $s(x_1^2)=4$ then we have $s(x_0^2)=2$ and $x_0=3$. Thus $s(x_1)=8$ and there is no such $x_1$, see Table 3 in~\cite{harelaishramstoll2011}.

\item If $s(x_1^2)=3$ then $x_0,x_1 \in E_3$ and we have $s(x_1)+s(x_0)\leq 8 <10$.

\item If $s(x_1^2)=2$ then $s(x_1)=2$ and $s(x_0)=8$. We conclude as in case (a). 
\end{enumerate}
\item Let $\boxed{k=11}$. 
\begin{enumerate}
\item If $s(x_1^2)=5$ then we have $s(x_0^2)=2$ and $x_0=3$. Thus $x_1$ satisfies
\begin{equation}\label{hop}
  s(3x_1)=4,\qquad s(x_1)=9.  
\end{equation}
A machine calculation shows that there is no solution of~(\ref{hop}) with $x_1\leq 2^{17}$ that also satisfies $s(x_1^2)=5$. \Cref{s(3x)=4} now states that if there were a solution of~(\ref{hop}) with $x_1> 2^{17}$ then \begin{align*}
\mathsf{x}_1&\in\{\tL \tO^\alpha (\tL \tO)^6 \tL \tL:\; \alpha\geq 1\}\;\cup\; \{
(\tL \tO)^6 \tL \tL \tO^\alpha \tL:\; \alpha\geq 1\}\\
&\qquad \cup\; \{(\tL \tO)^{\ell_1} \tL \tL \tO^\alpha (\tL \tO)^{\ell_2}\tL \tL:\; \ell_1+\ell_2=4,\; \ell_1, \ell_2\geq 0,\; \alpha \geq 1\}.
\end{align*}
A direct computer search shows that none of the above forms satisfies $s(x_1^2)=5$ (note that for sufficiently large $\alpha$ the sum of digits of the above forms stabilizes since the blocks in the square do not interfere anymore, so this is a finite verification.) Alternatively, we could also check the solutions of~(\ref{hop}) via the algorithms in the Section~\ref{sec-45}.
 
\item If $s(x_1^2)=4$. We have $s(x_0^2)=3$ and $x_0\in E_3$. Thus $s(x_1) \in\{ 7,8,9\}$. Algorithm \texttt{max-integer} shows that $E_{4,\lambda}$ is empty for $\lambda\in\{7,8,9\}$ which allows to conclude.
\end{enumerate}
\end{itemize}
\medskip

\item \underline{$s(x_1\cdot x_0)=5$.} \begin{itemize}
    \item If $\boxed{k=9}$, then we have $s(x_1^2)=s(x_0^2)=2$ and $x_1=x_0=3$. Hence $s(x_1)+s(x_0)=4<9$.
    \item If $\boxed{k=10}$, then we have $\max(s(x_1^2),s(x_0^2))=3$ and $\min(s(x_1^2),s(x_0^2))=2$ and $s(x_1)+s(x_0)\leq 6<10$.
    \item Let $\boxed{k=11}$. \begin{enumerate}
\item If $s(x_1^2)=4$, then we have $s(x_0^2)=2$ and $x_0=3$. Thus $s(x_1)=9$ and we conclude since $E_{4,9}$ is empty.

\item If $s(x_1^2)=3$, then we have $x_0,x_1 \in E_3$ and we have $s(x_1)+s(x_0)\leq 8 <11$.
\end{enumerate}
\end{itemize} 

\medskip
\item \underline{$s(x_1\cdot x_0)=6$.} \begin{itemize}
    \item If $\boxed{k=10}$, then we have $s(x_1^2)=s(x_0^2)=2$ and $x_1=x_0=3$. Then $s(x_1)+s(x_0)=4<10$.
    \item If $\boxed{k=11}$, then we have $\max(s(x_1^2),s(x_0^2))=3$ and $\min(s(x_1^2),s(x_0^2))=2$ and again $s(x_1)+s(x_0)\leq 6<11$.
\end{itemize} 
\medskip

\item \underline{$s(x_1\cdot x_0)=7$.} Here, necessarily $\boxed{k=11}$, thus we have $s(x_1^2)=s(x_0^2)=2$ and $x_1=x_0=3$. Then $s(x_1)+s(x_0)=4<11$.
\end{enumerate}
\bigskip

The proof of \Cref{main_lemma} is therefore complete for the case $m=1$. 

\subsection{The case $m=2$}

Here, we change our strategy and use the interference graph given in \Cref{fig m=2}. There are five independents summands $y_{2,2}, y_{2,1}, y_{2,0}, y_{1,0},y_{0,0}$ each contributing to $(n)_2$ with at least one $\tL$-bit and only $y_{2,0}$ may interfere with $y_{1,1}$. None of these five summands can contribute with more than $(k-4)$ $\tL$-bits since $s(n^2)=k$. 

We distinguish the following cases.

\begin{enumerate}
\item $\underline{x_2=1.}$ 
    \begin{itemize}
    \item If $\boxed{k=9}$ then $s(y_{2,1})=s(x_1)$ and $x_1$ has at most five $\tL$-bits. Hence $x_0\neq 1$ and $s(y_{1,0})\geq 2$. Therefore,
    $$1+s(x_1)+s(x_0)=9\geq  1+s(x_1)+1+2+s(x_0^2).$$
    Thus $s(x_0^2)\leq s(x_0)-3$. This condition has no solution when $2\leq s(x_0)\leq 7$, see~\cite{harelaishramstoll2011}.
    \item If $\boxed{k=10}$ then $s(y_{2,1})=s(x_1)$ and $x_1$ has at most six $\tL$-bits. Hence $x_0\neq 1$ and $s(y_{1,0})\geq 2$. As before, we get $s(x_0^2)\leq s(x_0)-3$. The only solution to this is $\mathsf{x}_0=\tL\tL\tL\tO\tL\tL\tL\tL\tL$ and $x_1=1$, see~\cite{harelaishramstoll2011}, Table 3. But now $y_{1,0}=x_0$ has $8>6$ $\tL$-bits, which is too many.
    \item If $\boxed{k=11}$ then $s(y_{2,1})=s(x_1)$ and $x_1$ has at most seven $\tL$-bits, thus $x_0\neq 1$ and $s(y_{1,0}) \geq 2$. Again, $2\leq s(x_0^2)\leq s(x_0)-3$ and the only solution for $s(x_0)\leq 8$ is $\mathsf{x}_0=\tL\tL\tL\tO\tL\tL\tL\tL\tL$. We get $\mathsf{x}_1 = \tL\tO^{\alpha}\tL$ for some $\alpha \geq 0$. But this implies $s(y_{1,0}) \geq 7$ and  this contribution is too large to $s(n^2)=11$ to hold since $s(y_{0,0})=5$. In fact, $s(y_{1,0})$ is constant for $\alpha\geq 9$ so we can check the values for $0\leq \alpha\leq 9$ directly. If there were any such solution for $s(x_0)=9$, then $x_1=1$ and so $s(y_{1,0})=9>11-4$. Again, this contradicts with $s(n^2) = 11$.
    \end{itemize}
    \medskip
\item $\underline{x_0=1}.$ This is symmetric to the first case. 
    \medskip
\item \underline{$x_2\neq 1$ and $x_0 \neq 1$.} The four summands $y_{2,2}, y_{2,1}, y_{1,0},y_{0,0}$ contain more than one $\tL$-bit. Note also that $x_1\ne 1$. In fact, if $x_1=1$, then $s(n)=s(x_2)+1+s(x_0)\leq s(y_{2,1})+s(y_{2,0})+s(y_{1,0})<s(n^2)$. \begin{itemize}
    \item If $\boxed{k=9}$, then all these summands contain two $\tL$-bits and therefore $x_2=x_0=3$. The fact that $s(x_1)=5$ and $s(x_2x_1)=2$, together with Lemma~\ref{useful_1}, implies that $\mathsf{x}_1=\tL\tO\tL\tO\tL\tO\tL\tL$. Then the summand $\mathsf{y}_{1,1}=\tL\tL\tL\tO\tO\tL\tO\tO\tO\tL\tL\tL\tO\tO\tL$ contributes with more than one $\tL$-bit to $(n^2)_2$. Indeed if we shift $\mathsf{y}_{1,1}=\tL\tL\tL\tO\tO\tL\tO\tO\tO\tL\tL\tL\tO\tO\tL$ against $\mathsf{y}_{2,0}=\tL\tO\tO\tL$ and add the terms, the result always has more than one $\tL$-bit (this can be checked by a simple \texttt{for}-loop, see the discussion in Section~\ref{sec-aux}).
    \item If $\boxed{k=10}$, then at least three of the summands contain exactly two $\tL$-bits. By symmetry, we may assume that $y_{2,2}$ and $y_{2,1}$ contain exactly two $\tL$-bits. This implies $x_2=3$ and $\mathsf{x}_1=(\tL\tO)^{\alpha}\tL\tL$ for some $\alpha\geq 0$, depending on $x_0$. \begin{enumerate}
\item If $s(y_{0,0})=2$ then $x_0=3$ and $\mathsf{x}_1=\tL\tO\tL\tO\tL\tO\tL\tO\tL\tL$. Thus we have  $\mathsf{y}_{1,1}=\tL\tL\tL\tO\tO\tO\tL\tL\tL\tL\tO\tO\tO\tL\tL\tL\tO\tO\tL$ which may interfere with $\mathsf{y}_{2,0}=\tL\tO\tO\tL$. This interference contributes with more than two $\tL$-bits. 
\item If $s(y_{0,0})=3$ then the possible choices of $\mathsf{x}_0$ are $\tL\tO\tL\tL\tL$, $\tL\tL\tL$ and $\tL\tO^{\beta}\tL$ for some $\beta\geq 1$, see Szalay's result~(\ref{s(n2)=3}). Therefore, $4\leq s(x_1)\leq 6$ and we can easily check that in all cases $y_{1,0}$ has more than two $\tL$-bits. 
\end{enumerate}

\item If $\boxed{k=11}$, the four summands $y_{2,2}, y_{2,1}, y_{1,0}, y_{0,0}$ contain more than one $\tL$-bit, thus at least two of them contain exactly two $\tL$-bits. We now distinguish the cases regarding the number of summands with two $\tL$-bits.

\smallskip First, suppose that there are at least three terms among $y_{2,2}, y_{2,1}, y_{1,0}, y_{0,0}$ that contain exactly two $\tL$-bits. By symmetry we can assume that $y_{2,2}$ and $y_{2,1}$ contain exactly two $\tL$-bits. This implies $\mathsf{x}_2=\tL\tL$ and $\mathsf{x}_1=(\tL\tO)^{\alpha}\tL\tL$ for some $\alpha \geq 0$. \begin{enumerate}
\item If $s(y_{0,0}) = 2$ then $\mathsf{x}_0=\tL\tL$, $\mathsf{x}_1=(\tL\tO)^{5}\tL\tL$ and $$\mathsf{y}_{1,1} = \tL\tL\tL\tO\tO\tO\tL\tL\tL\tO\tO\tL\tL\tL\tO\tO\tO\tL\tL\tL\tO\tO\tL$$ which may interfere with $\mathsf{y}_{2,0} = \tL\tO\tO\tL$. This interference has more than three $\tL$-bits.
\item If $s(y_{0,0}) = 3$ then the possible choices for $\mathsf{x}_0$ are $\tL\tO\tL\tL\tL$, $\tL\tL\tL$, and $\tL\tO^{\beta}\tL$ for some $\beta \geq 1$. Therefore, $s(x_1)\geq 5$ and in all cases $y_{1,0}$ has more than two $\tL$-bits.
\item If $s(y_{0,0}) = 4$ then, since $s(x_0)\leq 7$, the possible choices for $\mathsf{x}_0$ are $\tL\tO\tL\tL\tL\tL$, $ \tL\tL\tO\tL\tL\tL\tL$, $\tL\tL\tL\tL$, and $\tL\tL\tO\tL$ (see Table 1 and Table 3 in~\cite{harelaishramstoll2011}). Therefore, $s(x_1) \geq 3$ and in all cases $y_{1,0}$ has more than two $\tL$-bits.
\end{enumerate}

\smallskip 

Suppose now that there are exactly two among $y_{2,2}, y_{2,1}, y_{1,0}, y_{0,0}$ that contain two $\tL$-bits. It is then sufficient to discuss the following cases: \begin{enumerate}
\item  $y_{2,2}$ and $y_{1,0}$ (or analogously $y_{2,1}$ and $y_{0,0}$) contain exactly two $\tL$-bits. Then $\mathsf{x}_2=\tL\tL$ and $s(y_{0,0})=3$ so that $x_0 \in E_3$. Hence, $5\leq s(x_1) \leq 7$. We now consider $y_{2,1}=3x_1$ and $s(3x_1)=3$. 



By Lemma~\ref{s(3x)=3}, the only solutions for $\mathsf{x}_1$ are $\tL(\tO \tL)^{\ell_1}(\tL \tO)^{\ell_2} \tL \tL$ for some $\ell_1,\ell_2 \geq 0$ such that $2 \leq \ell_1+\ell_2 \leq 4$. 

Therefore, all the possible values for $\mathsf{y}_{2,0}$ and $\mathsf{y}_{1,1}$ are:

$$\begin{tabular}{l|l}
     $\mathsf{y}_{2,0}$ & $\mathsf{y}_{1,1}$ \\ \hline \hline 
     $\tL\tO\tL\tO\tL$ & $\tL\tL\tL\tO\tL\tL\tO\tO\tL\tO\tO\tO\tL$  \\ \hline
     $\tL\tO\tO\tO\tL\tO\tL$ & $\tL\tO\tO\tO\tO\tO\tO\tL\tO\tL\tL\tO\tO\tL$ \\\hline
     $\tL\tL\tO^{t-2}\tL\tL$, $t \geq 2$ & $\tL\tO\tL\tL\tO\tO\tL\tO\tL\tL\tL\tO\tO\tL$ \\ \hline 
     & $\tL\tL\tL\tO\tO\tL\tO\tL\tL\tL\tO\tO\tL\tO\tO\tO\tL$ \\ \hline 
     & $\tL\tL\tL\tO\tL\tO\tL\tL\tO\tO\tL\tO\tL\tL\tO\tO\tL$ \\ \hline 
     & $\tL\tO\tO\tO\tO\tO\tO\tO\tL\tO\tL\tO\tL\tL\tL\tO\tO\tL$ \\ \hline 
     & $\tL\tO\tO\tO\tO\tO\tO\tO\tL\tO\tL\tO\tL\tL\tL\tO\tO\tL$ \\ \hline 
     & $\tL\tL\tL\tO\tL\tO\tL\tO\tL\tL\tO\tL\tO\tL\tO\tL\tL\tL\tO\tO\tL$ \\ \hline 
     & $\tL\tL\tL\tO\tL\tO\tL\tO\tL\tL\tO\tL\tO\tL\tO\tL\tL\tL\tO\tO\tL$ \\ \hline 
     & $\tL\tL\tL\tO\tL\tO\tL\tO\tL\tL\tO\tL\tO\tL\tO\tL\tL\tL\tO\tO\tL$ \\ \hline 
     & $\tL\tO\tO\tO\tO\tO\tO\tO\tL\tO\tO\tO\tO\tO\tO\tO\tL\tL\tL\tO\tO\tL$ \\ \hline 
     & $\tL\tO\tO\tO\tO\tO\tO\tO\tL\tO\tO\tO\tO\tO\tO\tO\tL\tL\tL\tO\tO\tL$ \\ \hline 
\end{tabular}$$
\medskip

Therefore, in each of the above possibilities the terms $y_{1,1}$ and $y_{2,0}$ will contribute with more than one $\tL$-bit.


\item $y_{2,2}$ and $y_{0,0}$ contain exactly two $\tL$-bits. Then $\mathsf{x}_2=\mathsf{x}_0 = \tL\tL$ so $s(x_1) = 7$. Thus the contribution of interference of $y_{2,0}$ with $y_{1,1}$ is only one $\tL$-bit if and only if $(x_1^2)_2$ is of the form $\tL\cdots\tL\tO\tL\tL\tL$. This implies that $x_1^2 \equiv 7 \pmod 8$, which is not possible for any odd integer $x_1$.

\item $y_{2,1}$ and $y_{1,0}$ contain exactly two $\tL$-bits. Then $y_{2,2}$ and $y_{0,0}$ contain exactly three $\tL$-bits thus $x_0,x_2 \in E_3$. The summand $y_{2,0}$ which might interfere with $y_{1,1}$ has to be one $\tL$-bit.

The following forms for $y_{2,0}$ contradict the fact that $y_{1,1} \equiv 1 \pmod 8$: $$\begin{tabular}{l|l|l}
     $\mathsf{y}_{2,0}$ & $(x_2,x_0)$ & Requested form for $\mathsf{y}_{1,1}$  \\ \hline \hline
     $\tL\tL\tO\tO\tO\tL$ &  (7,7) & $(*)\tL\tL$ \\ \hline
     $\tL\tO\tO\tO\tO\tL\tO\tO\tO\tL$ & (23,23) & $(*)\tL\tL$  \\ \hline 
     $\tL\tO\tL\tO\tO\tO\tO\tL$ & (7,23) & $(*)\tL\tL$  \\ \hline
     $\tL\tO\tO\tO\tL$ &  $(7,2^2+1)$ & $(*)\tL\tL$ \\ \hline
     $\tL\tL\tL\tO\tO\tL\tL$ &  $(23,2^2+1)$ & $(*)\tL\tO\tL$ \\ \hline
     - &$(2^{t_2}+1,2^{t_0}+1),t_2,t_0\geq 2$ & $(*)\tL\tL$ \\
\end{tabular}$$

The remaining forms are more complicated because the same argument does not work. 

$$\begin{tabular}{l|l|l}
     $\mathsf{y}_{2,0}$ & $(x_2,x_0)$  & Requested form for $\mathsf{y}_{1,1}$ \\ \hline \hline
     $\tL\tL\tO\tO\tL\tL\tL\tL$ &  $(23,2^3+1)$ & $\tL\cdots\tL\tO\tO\tL\tL\tO\tO\tO\tL$  \\ \hline
     $\tL\tL\tO\tO\tO\tO\tL\tL\tL$ &  $(23,2^4+1)$ & $\tL\cdots\tL\tO\tO\tL\tL\tL\tL\tO\tO\tL$  \\ \hline 
     $\tL\tL\tL\tO^{t-3}\tL\tL\tL$ & $(7,2^t+1), t\geq 3$ & $\tL\cdots\tL\tO\tO\tO\tL^{t-3}\tO\tO\tL$ \\ \hline
     $\tL\tO\tL\tL\tL\tO^{t-5}\tL\tO\tL\tL\tL$ & $(23,2^t+1), t\geq 5$ &  $\tL\cdots\tL\tO\tL\tO\tO\tO\tL^{t-5}\tO\tL\tO\tO\tL$  \\ 
\end{tabular}$$

The first case implies that $x_1^2=2^8(2^{\lambda}-1)+49$ for some $\lambda \geq 1$, since $s(x_1)=5$ gives that $x_1^2=49$ is not possible. However, $\lambda$ is bounded since $s(x_1^2)\leq s(x_1)(s(x_1)+1)/2$, so $\lambda\leq 12$. Thus it is sufficient to check if the integers $2^8(2^{\lambda}-1)+49$, for $0 \leq \lambda \leq 12$ are perfect square, and it is not the case.

The second case is similar since we have $x_1^2=2^8(2^{\lambda}-1)+111$ for some $\lambda \geq 1$ and $s(x_1)=4$. We conclude in the same way as before. 

As for the third case, we have $x_1^2=2^{t+4}(2^{\lambda}-1)+2^3(2^{t-3}-1)+1$ for some $\lambda \geq 1$ and $s(x_1)=6$. Again, $t$ is bounded and we find the only solution $x_1=31$ for $t=3$. However, $s(y_{2,1})=s(7 \times 31)=5$ contradicts our first hypothesis in this case. The last case works in a same manner and we are done. 
\end{enumerate}
\end{itemize} 
\end{enumerate}

Therefore, the proof of \Cref{main_lemma} is complete for $m=2$. 

\subsection{The case $m=3$}

There are seven independent summands $y_{3,3},\ldots,y_{3,0},\ldots,y_{0,0}$ each contributing to $(n^2)_2$ with at least one $\tL$-bit, see \Cref{fig m=3}. The considerations for $k=9,10,11$ are slightly different here.

\medskip 

If $\boxed{k=9}$ then at least five of the seven independents summands have to contain only one $\tL$-bit.  

\begin{enumerate}
\item \underline{$x_3=x_0=1$.} We have $y_{3,2} =x_2$ and $y_{1,0} =x_1$. Therefore , $s(n^2)> s(y_{3,3})+s(y_{3,2})+s(y_{1,0})+s(y_{0,0})=s(x_3)+s(x_2)+s(x_1)+s(x_0)= 9$, a contradiction.

\medskip
\item \underline{$x_3 \neq 1$.} The summands $y_{3,3}$ and $y_{3,2}$ contain two $\tL$-bits while all the other summands only contribute with one $\tL$-bit. Thus, $x_3= 3$,  $x_1=x_0=1$, and $s(x_2) = 5$. Lemma~\ref{useful_1} shows that the only solution to $s(y_{3,2}) = 2$ is $\mathsf{x}_2=\tL\tO\tL\tO\tL\tO\tL\tL$.  However, the summand $\mathsf{y}_{2,2}=\tL\tL\tL\tO\tO\tL\tO\tO\tO\tL\tL\tL\tO\tO\tL$ can only interfere with either $\mathsf{y}_{3,1}=\tL\tL$ or $\mathsf{y}_{3,0}=\tL\tL$ and in both cases its contribution to $n^2$ is larger than one $\tL$-bit.

\medskip

\item  \underline{$x_0 \neq  1$.} This is symmetric to the previous case.

\end{enumerate}

\medskip

If $\boxed{k=10}$ then, among the seven independent summands $y_{3,3},\ldots,y_{3,0},\ldots,y_{0,0}$, at least four of them must contain exactly one $\tL$-bit and none of the summands can contribute with more than three $\tL$-bits. In fact, if there were a summand with more than three $\tL$-bits then all other six summands have to contribute with a single $\tL$-bit which is impossible. We now distinguish several cases.
\begin{enumerate}
\item \underline{$x_3=x_0=1$.} We have $y_{3,2}=x_2$ and $y_{1,0}=x_1$, therefore, $s(n^2)>s(y_{3,3})+s(y_{3,2})+s(y_{1,0})+s(y_{0,0})=10$, a contradiction.
\item \underline{$s(y_{3,3})=2$.} This implies $\mathsf{x}_3=\tL\tL$: the summand $y_{3,2}$ contains at least two $\tL$-bits, thus, one of $\mathsf{y}_{1,0}$ and $\mathsf{y}_{0,0}$ is $\tL$ which implies that $x_0=1$. \begin{enumerate}
\item \underline{$s(y_{3,2})=2$ and $s(y_{1,0})=2$.} As $y_{1,0}=x_1$ and $s(x_1)=2$, we obtain $\mathsf{x}_2=(\tL\tO)^{3}\tL\tL$ and $\mathsf{y}_{2,2}=\tL\tL\tL\tO\tO\tL\tO\tO\tO\tL\tL\tL\tO\tO\tL$. This summand may interfere with one of 
$$\mathsf{y}_{3,1}\in\{\tL \tL \tO^\gamma \tL \tL: \gamma \geq 0\}\cup \{\tL \tO\tO \tL\}$$ and $\mathsf{y}_{3,0}=\tL \tL$, in both cases its contribution will be more than one $\tL$-bit. Note that for $y_{3,1}$ we can use a \texttt{for}-loop over $\gamma$ to conclude.
\item \underline{$s(y_{3,2})=2$ and $s(y_{1,0})=1$.} We have $x_1=1$, $\mathsf{x}_2=(\tL\tO)^{4}\tL\tL$ and  $$\mathsf{y}_{2,2}=\tL\tL\tL\tO\tO\tO\tL\tL\tL\tL\tO\tO\tO\tL\tL\tL\tO\tO\tL.$$ This summand may interfere with one of $y_{3,1}=x_3=3$ and $y_{3,0}=x_3=3$, but its contribution will be more than two $\tL$-bits.
\item \underline{$s(y_{3,2})=3$.} Here again $x_1=1$ and the summand $y_{2,1}=x_2$ can only interfere with $\mathsf{y}_{3,0}=\tL\tL$ and they must add up to a power of 2. According to Lemma~\ref{s(3x)=3}, the system $s(3 x_2)=3$, $s(x_2)=6$ implies that \[ \mathsf{x}_2 \in \{\tL\tO\tL\tO\tL\tO\tL\tL\tL, \tL\tO\tL\tL\tO\tL\tO\tL\tL,\tL\tO\tL\tO\tL\tL\tO\tL\tL,\tL\tL\tO\tL\tO\tL\tO\tL\tL\}. \]

In any case, the interference between $y_{2,1}$ and $y_{3,0}=3$ would then again contribute with too many $\tL$-bits.
\end{enumerate}
\item \underline{$s(y_{3,3})=3$.} We have $x_3\in E_3$ and $s(y_{3,2})\geq 2$. Furthermore, we have $x_1=x_0=1$ because $y_{1,0}=1$. In any case $4\leq s(x_2)\leq 6$ the summand $y_{2,0}=x_2$ has to interfere with $y_{1,1}=1$ and add up to a power of $2$. Thus, $\mathsf{x}_2 =\tL^j$ for $4\leq j\leq 6$ and $\mathsf{y}_{2,2}=\tL^{j-1}\tO^j\tL$. This summand can only interfere with one of $y_{3,1}=x_3$ and $y_{3,0}=x_3$ and we see that the contribution of this summand is then again more than one $\tL$-bit in all cases.
\item \underline{$s(y_{0,0})\geq 2$}. These cases are symmetric to the previous two cases.
\end{enumerate}

\medskip

If $\boxed{k=11}$ then at least three of the independents summands have to contain one $\tL$-bit only. We distinguish between the cases: \begin{enumerate}
\item \underline{$x_3=x_0=1$.} We have $y_{3,2} =x_2$ and $y_{1,0} =x_1$, therefore, $s(n^2)>s(y_{3,3})+s(y_{3,2})+s(y_{1,0})+s(y_{0,0})=11$, again a contradiction.
\item \underline{$s(y_{3,3})=4$.} Since $s(x_3)\leq 8$ the possible values for $\mathsf{x}_3$ are $\tL\tO\tL\tL\tL\tL$, $\tL\tL\tO\tL\tL\tL\tL$, $\tL\tL\tL\tL$ and $\tL\tL\tO\tL$  (see~\cite{harelaishramstoll2011}). Also, since in this case $s(y_{3,2})\geq 2$, we have $x_1=x_0=1$. This implies $3\leq s(x_2)\leq 6$. The summand $y_{2,0}=x_2$ has to interfere with $y_{1,1}=1$ and the contribution has exactly one $\tL$-bit. Hence, $\mathsf{x}_2=\tL^{i}$ for $3\leq i \leq 6$ and $\mathsf{y}_{2,2} =\tL^{i-1}\tO ^{i}\tL$. This summand can only interfere with one of $y_{3,1}=x_3$ and $y_{3,0}=x_3$ and we can see that the contribution of this summand is more than one $\tL$-bit in each case.
\item \underline{$s(y_{3,3})=3$.} Possible forms of $\mathsf{x}_3$ are $\tL\tO\tL\tL\tL$, $\tL\tL\tL$ and $\tL\tO^{i}\tL$ for some $i \geq 1$. Furthermore, $x_0=1$ since one of $y_{0,0}$ and $y_{1,0}$ is $1$. \begin{enumerate}
\item If $s(x_1)=2$ then in any case $4\leq s(x_2)\leq 6$. The summand $y_{2,0} =x_2$ has to interfere with $y_{1,1}$ and contributes with one $\tL$-bit. Since $\mathsf{x}_1=\tL\tO^{j}\tL$ for some $j \geq 0$, we have $\mathsf{y}_{1,1}$ equals to $\tL\tO\tO\tL$ or $\tL\tO^{j-1}\tL\tO^{j+1}\tL$ for some $ j\geq 1$. A computation regarding the possible interference of the summands shows that we have one of the following cases: 
\bigskip

\begin{center}
\begin{tabular}{c|c}
$s(x_2)$ & $\mathsf{x}_2$  \\ \hline\hline
4 &  $\{\tL\tO\tL\tL\tL, \tL\tO\tO\tL\tL\tL\}$  \\  \hline
5 & $\{\tL\tL\tO\tL\tL\tL,\tL\tL\tO\tO\tL\tL\tL,\tL\tO\tL\tL\tL\tL\}$  \\\hline 
6 & $\{\tL\tL\tL\tO\tL\tL\tL, \tL\tL\tL\tO\tO\tL\tL\tL,\tL\tO\tL\tO\tL\tL\tL\tL\}$  \\ 
\end{tabular}
\end{center}


\bigskip

By computing every value of $y_{2,2}$ we finally have to check the possible interference of $y_{2,2}$ with $y_{3,1}$ or $y_{3,0}=x_3$. In each case the result has more than one $\tL$-bit and this leads to a contradiction.
\item If $x_1=1$ then $5\leq s(x_2)\leq 7$. First, suppose that the summand $y_{2,0}=x_2$ interferes with $y_{1,1}=1$ and adds up to a power of $2$. Then $\mathsf{x}_2=\tL^{j}$ for $5\leq j\leq 7$ and $s(y_{3,2})\geq 3$. The term $\mathsf{y}_{2,2} = \tL^{j-1}\tO^{j}\tL$ can only interfere with one of $y_{3,1}=x_3$ and $y_{3,0}=x_3$ and the contribution is again more than one $\tL$-bit.  If $y_{2,0}=x_2$ interferes with $y_{1,1}=1$ and gives a contribution of the form $\tL\tO^{j}\tL$, then $\mathsf{x}_2$ is of the form $\tL\tL\tL\tL\tO^{j-4}\tL$,  $\tL\tL\tL\tL\tL\tO^{j-5}\tL$ or $\tL\tL\tL\tL\tL\tL\tO^{j-6}\tL$. We compute $y_{2,2}$ in all theses cases and conclude that this summand contributes with more than one $\tL$-bit when it interacts with $x_3$. 
\end{enumerate} 
\item \underline{$s(y_{3,3})=2$.} We have $\mathsf{x}_3=\tL\tL$. 

\smallskip 

First, we assume that $x_0=1$. We have the following cases: 
\bigskip

\begin{center}
\begin{tabular}{c|c|c|c}
$s(y_{3,2})$ & $s(y_{1,0})$ & $\mathsf{x}_2$ & $\mathsf{y}_{2,2}$  \\ \hline \hline
2 & 3 & $(\tL\tO)^{3}\tL\tL$ & $\tL\tL\tL\tO\tO\tL\tO\tO\tO\tL\tL\tL\tO\tO\tL$ \\ \hline
2 & 2 & $(\tL\tO)^{4}\tL\tL$ & $\tL\tL\tL\tO\tO\tO\tL\tL\tL\tL\tO\tO\tO\tL\tL\tL\tO\tO\tL$ \\ \hline
2 & 1 & $(\tL\tO)^{5}\tL\tL$ & $\tL\tL\tL\tO\tO\tO\tL\tL\tL\tO\tO\tL\tL\tL\tO\tO\tO\tL\tL\tL\tO\tO\tL$ \\ \hline 
3 & 2 & $(\tL\tO)^{4}\tL\tL$ & $\tL\tL\tL\tO\tO\tO\tL\tL\tL\tL\tO\tO\tO\tL\tL\tL\tO\tO\tL$ \\ \hline 
3 & 1 & $(\tL\tO)^{5}\tL\tL$ & $\tL\tL\tL\tO\tO\tO\tL\tL\tL\tO\tO\tL\tL\tL\tO\tO\tO\tL\tL\tL\tO\tO\tL$ \\ \hline 
4 & 1 & $(\tL\tO)^{5}\tL\tL$& $\tL\tL\tL\tO\tO\tO\tL\tL\tL\tO\tO\tL\tL\tL\tO\tO\tO\tL\tL\tL\tO\tO\tL$ \\ 
\end{tabular}
\end{center}

\bigskip

In each case, $y_{2,2}$ may interfere with one of $y_{3,1}=x_3\cdot x_1$ and $\mathsf{y}_{3,0}=\tL\tL$. In both cases; this interference is always more than one $\tL$-bit due to the submultiplicativity property of the sum of digits function. Indeed, this fact ensure that $s(x_3\cdot x_1)\leq s(x_3)\cdot s(x_1) \leq 6$ and it is not sufficient to cancel all the inner $\tL$-bits in $y_{2,2}$. 

\smallskip

Secondly, assume that $x_0 \neq 1$. Then, $s(y_{0,0})=s(y_{1,0})=s(y_{3,2})=2$ and so $\mathsf{x}_0= \tL\tL$, $\mathsf{x}_1=(\tL\tO)^{i}\tL\tL$ and $\mathsf{x}_2=(\tL\tO)^{j}\tL\tL$ for some positive integers $i$ and $j$ such that $i+j=3$. (The case $x_1=1$ and/or $x_2=1$ is easier and can be treated in a similar fashion.) By symmetry we may suppose $i<j$. 
\begin{enumerate}
\item If $j=3$, then $\mathsf{x}_2=\tL\tO\tL\tO\tL\tO\tL\tL$ and $\mathsf{x}_1=\tL\tL$, and we have  $\mathsf{y}_{2,2}=\tL\tL\tL\tO\tO\tL\tO\tO\tO\tL\tL\tL\tO\tO\tL$. This summand may interfere with one of $\mathsf{y}_{3,1}=\tL\tO\tO\tL$ and $\mathsf{y}_{3,0}=\tL\tO\tO\tL$, but this contribution is always more than one $\tL$-bit. 
\item If $j=2$, then $\mathsf{x}_2=\tL\tO\tL\tO\tL\tL$ and $\mathsf{x}_1=\tL\tO\tL\tL$. Then we have  $\mathsf{y}_{2,2}=\tL\tL\tL\tO\tO\tL\tL\tL\tO\tO\tL$. This summand may interfere with one of $\mathsf{y}_{3,1}=\tL\tO\tO\tO\tO\tL$ and $\mathsf{y}_{3,0} = \tL\tO\tO\tL$, but this contribution is always more than one $\tL$-bit.
\end{enumerate}
\end{enumerate}

Therefore, the proof of \Cref{main_lemma} is complete for $m=3$. 

\subsection{The case $m=4$}

There are nine independent summands $y_{4,4},\ldots,y_{4,0},\ldots,y_{0,0}$ each contributing to $n^2$ with at least one $\tL$-bit and at most $(k-8)$ $\tL$-bits, see \Cref{fig m=4}. While there will be more restrictions compared to the previous cases due to the fact that there are more independent terms, the downside is that we have now three levels of interaction in the interference graph.

\bigskip

Let $\boxed{k=9}$. This is the easiest case since it implies that all of these summands contribute to $n^2$ with exactly one $\tL$-bit. We have $y_{4,4}=y_{4,3}=y_{1,0}=y_{0,0}=1$ it implies that $x_4 =x_3 =x_1 =x_0 =1$ and $s(x_2)=5$. The summand $y_{4,2}=x_2$ can only interfere with $y_{3,3}=1$ and the result of adding these two summands has to be a power of 2; this is only possible if $\mathsf{x}_2 = \tL\tL\tL\tL\tL$. Now, the summand $\mathsf{y}_{2,2} = \tL\tL\tL\tL\tO\tO\tO\tO\tL$ can interfere with two of the summands $y_{3,1} , y_{4,1} , y_{3,0}$ and $y_{4,0}$ . As each of these four summands is $1$, this contribution to $n^2$ has more than one $\tL$-bit. This concludes this case.

\medskip 

For $\boxed{k=10}$, among the nine independant summands, only one can contribute with two $\tL$-bits. We immediately obtain that $x_0=x_4=1$ and one of $x_1$ and $x_3$ is $1$.

\begin{enumerate}
    \item \underline{$x_3 \neq 1$} (or symmetrically $x_1 \neq 1$). We see that $s(x_3)=s(y_{4,3})=2$ and $s(x_2)=5$. The factor $y_{2,0}=x_2$, which can only interfere with $y_{1,1}=1$, has to contribute with one $\tL$-bit only, therefore, $\mathsf{x}_2=\tL\tL\tL\tL\tL$. Now, the summand $\mathsf{y}_{2,2}=\tL\tL\tL\tL\tO\tO\tO\tO\tO\tL$ can interfere with two of the summands $y_{3,1}=x_3, y_{4,1}=1, y_{3,0}=x_3$, and $y_{4,0}=1$. This contribution to $n^2$ has more than one $\tL$-bit.
    \item \underline{$x_3=x_1=1$.} We have $s(x_2)=6$ and the factor $y_{2,0}=x_2$, which can only interfere with $y_{1,1}=1$, has to contribute with one $\tL$-bit only. Indeed if it contributes with more than $\tL$-bit, by symmetry $y_{4,2}$ will interfere with $y_{3,3}=1$ with more than one $\tL$-bit and we have a contradiction.  
    In fact, if $\mathsf{x}_2\ne\tL\tL\tL\tL\tL\tL$, then the factor $y_{2,0}=x_2$ (resp. $y_{4,2}=x_2$), which can only interfere with $y_{1,1}=1$ (resp. $y_{3,3}=1$) contributes contribute with more than one $\tL$-bit. Therefore, $\mathsf{x}_2=\tL\tL\tL\tL\tL\tL$.
    The summand $\mathsf{y}_{2,2}=\tL\tL\tL\tL\tL\tO\tO\tO\tO\tO\tO\tL$ can interfere with two of the summands $y_{3,1}=1$, $ y_{4,1}=1$, $y_{3,0}=1$, and $y_{4,0}=1$ and its contribution to $n^2$ has to be at most two $\tL$-bits. There is a solution for the given conditions as the summands may be shifted against each other. However, as we will see next, the contradiction arises when we try to find a solution for $\hat{\ell_1},\ldots,\hat{\ell_4}$, where $$\hat{\ell_j}=\sum \limits_{i=1}^j \ell_i + |\mathsf{x}_{i-1}|$$ (we refer to Section~\ref{sec-aux} for the notation). The following pairs of summands have to interfere with each other such that their contribution has one $\tL$-bit only: $$(y_{2,0}, y_{1,1}),\; (y_{3,0}, y_{2,1}),\; (y_{4,1}, y_{3,2}),\; (y_{4,2}, y_{3,3}).$$ More precisely, their least significant bits have to align: 
\begin{equation*}
\left\{
    \begin{array}{lll}
        \hat{\ell_{2}} +1=2\hat{\ell_{1}} \implies \hat{\ell_{2}} =2\hat{\ell_{1}} -1. \\
        \hat{\ell_{3}} +1=\hat{\ell_{2}} +\hat{\ell_{1}} +1 \implies \hat{\ell_{3}} =3\hat{\ell_{1}} -1. \\
        \hat{\ell_{4}} +\hat{\ell_{1}} +1=\hat{\ell_{3}} +\hat{\ell_{2}} +1 \implies \hat{\ell_{4}} =4\hat{\ell_{1}} -2. \\
        \hat{\ell_{4}} +\hat{\ell_{2}} +1=2\hat{\ell_{3}}.
    \end{array}
\right.
\end{equation*}
Recall that $\hat{\ell_{0}}=0$. Thus, the summands $y_{2,2}, y_{3,1}$, and $y_{4,0}$ align as\begin{align*}
\tL\tL\tL\tL\tL\tO\tO\tO\tO\tO&\tO\tL \\ \tL& \\ &\tL
\end{align*}
and their contribution to $n^2$ is $\tL\tL\tL\tL\tL\tO\tO\tO\tO\tL\tL\tL$ which is too much.
\end{enumerate}

\bigskip

For $\boxed{k=11}$ we distinguish the following cases: 

\medskip

\begin{enumerate}
\item \underline{$x_4 \neq 1$.} The summands $y_{4,4}$ and $y_{4,3}$ have to contain two $\tL$-bits and all other summands only contribute with one $\tL$-bit. Thus $x_4=3$ and $x_1=x_0=1$. As $s(y_{4,3})=2$, we have $\mathsf{x}_3=(\tL\tO)^{i}\tL\tL$ for some $0\leq i \leq 4$. Furthermore $y_{2,0}=x_2$ can only interfere with $y_{1,1}$ and this summand has to contribute with only one $\tL$-bit so that $\mathsf{x}_2=\tL^{k}$ for some $1 \leq k \leq 5$. Yet the summand $y_{4,2}$ can only interfere with $y_{3,3}$. The statement $\sum_{0\leq i \leq 4}s(x_i)=11$ implies $i+k=5$. Thus the possible values for the couple $(y_{4,2},y_{3,3})$ are: 

\bigskip
$$\begin{tabular}{l|l}
     $\mathsf{y}_{4,2}$ & $\mathsf{y}_{3,3}$   \\ \hline \hline
     $\tL\tO\tL\tL\tL\tO\tL$ & $\tL\tO\tO\tL$ \\ \hline
     $\tL\tO\tL\tL\tO\tL$ & $\tL\tL\tL\tL\tO\tO\tL$ \\\hline
     $\tL\tO\tL\tO\tL$ & $\tL\tL\tL\tO\tO\tL\tL\tL\tO\tO\tL$ \\\hline
     $\tL\tO\tO\tL$ & $\tL\tL\tL\tO\tO\tL\tO\tO\tO\tL\tL\tL\tO\tO\tL$ \\\hline
     $\tL\tL$ & $\tL\tL\tL\tO\tO\tO\tL\tL\tL\tL\tO\tO\tO\tL\tL\tL\tO\tO\tL$ \\
\end{tabular}$$  

\bigskip

This gives a contribution to $n^2$ that has more than one $\tL$-bit.
\medskip

\item \underline{$x_0 \neq 1$.} This is symmetric to the previous case.  

\medskip

\item \underline{$x_4=x_0=1$.} By symmetry we may assume that $s(x_3) \geq s(x_1)$. As each of the nine independent summands has to contribute to $n^2$ with at least one $\tL$-bit, by inspection of $y_{4,3}=x_3$ we have $s(x_3)\leq 3$. By an inspection of $y_{1,0}=x_1$, we have in the same way $s(x_1)\leq 3$. If $s(x_3)=3$, then $s(x_1)=1$. We finally have to discuss the following four cases:

\bigskip

\begin{center}
\begin{tabular}{c|c|c}
  $s(x_3)$ & $s(x_1)$  & $s(x_2)$ \\ \hline \hline
  3 & 1 & 5 \\ \hline 
  2 & 2 & 5 \\ \hline 
  2 & 1 & 6 \\ \hline 
  1 & 1 & 7 \\ 
\end{tabular}
\end{center}

\bigskip

For the \underline{first line} in this table, except for $y_{4,3}=x_3$, all contributions have to be exactly one $\tL$-bit. Since $y_{2,0}=x_2$ can only interfere with $y_{1,1}=1$, we have $\mathsf{x}_2=\tL\tL\tL\tL\tL$. Furthermore, $y_{4,2}=x_2$ can only interfere with $y_{3,3}=x_3^2$ and this contribution has to be a single $\tL$-bit. This happens if and only if $(x_3^2)_2=\tL\cdots\tL\tO\tO\tO\tO\tL$, i.e. $x_3^2$ correspond to the integer $1+2^5(2^{\lambda}-1)$ for some $\lambda\geq 1$ . We can solve this such as in the case $m=2,k=11$ (c), or more directly with the following calculation: Write $x_3=1+2^a+2^b$ for some $0<a<b$. Then \begin{align*}
x_3^2&=1+2^{a+1}+2^{b+1}+2^{2a}+2^{a+b+1}+2^{2b}, \\
&=1+2^{a+1}(1+2^{b-a}+2^{a-1}+2^{b}+2^{2b-a-1}).
\end{align*}

Thus we have $a=4$ and all other power of $2$ have to be consecutive. This implies in particular $b-a=2$ and $b=3$. This is not possible since $a<b$.

\medskip 

For the \underline{second line}, $\mathsf{x}_1=\tL\tO^{i}\tL$ for some $i\geq 0$. Thus $\mathsf{y}_{1,1}=\tL\tO\tO\tL$ for $i=0$ or $\tL\tO^{i-1}\tL\tO^{i+1}\tL$ for $i\geq 1$. Again, the summand $y_{2,0}=x_2$ can only interfere with $y_{1,1}$ and this contribution has more than one $\tL$-bit except for the cases $(\mathsf{x}_1,\mathsf{x}_2)=(\tL\tL,\tL\tL\tO\tL\tL\tL)$ and $(\tL\tO\tL,\tL\tL\tO\tO\tL\tL\tL)$. For $i\geq 2$, having a contribution of one $\tL$-bit implies that $s(x_2) \geq 6$, and the blocks of $\tO$-bits of $y_{1,1}$ are too large to be covered. Yet $y_{4,2}=x_2$ can only interfere with $y_{3,3}$ with only one $\tL$-bit and we have the same result as before. The value of $x_2$ sets the values of $x_1$ and $x_3$, and we have to study the two following cases: 

\begin{enumerate}
\item If $\mathsf{x}_2=\tL\tL\tO\tL\tL\tL$, then we have $\mathsf{x}_1=\mathsf{x}_3=\tL\tL$. This implies $\mathsf{y}_{3,2}=\tL\tO\tL\tO\tO\tL\tO\tL$ and $y_{3,2}$ can interfere with $y_{4,1}=x_1$ and $y_{4,0}=x_0$. In all cases the contribution is more than one $\tL$-bit.

\item If $\mathsf{x}_2=\tL\tL\tO\tO\tL\tL\tL$, then we have $\mathsf{x}_1=\mathsf{x}_3=\tL\tO\tL$. This implies $\mathsf{y}_{3,2}=\tL\tO\tO\tO\tO\tO\tO\tO\tL\tL$  and $y_{3,2}$ can interfere with $y_{4,1}=x_1$ and $y_{4,0}=x_0$. In all cases the contribution is more than one $\tL$-bit.
\end{enumerate}

\medskip

For the \underline{third line}, we have that $y_{2,0}=x_2$ can only interfere with $y_{1,1}=x_1^2=1$ with at most two $\tL$-bits. We write $\mathsf{x}_3={ \tL\tO^{i}\tL}$ for some $i \geq 0$. We distinguish two cases according to this contribution.
\begin{enumerate}
\item If this contribution has only one $\tL$-bit then we have $\mathsf{x}_2=\tL\tL\tL\tL\tL\tL$. Yet $y_{4,2}=x_2$ can only interfere with $y_{3,3}=x_3^2$ and the contribution has at most two $\tL$-bits. This implies $\mathsf{x}_3 \in \{\tL\tL,\tL\tO\tO\tL,\tL\tO\tO\tO\tO\tL,\tL\tO\tO\tO\tO\tO\tO\tL\}$, i.e $i \in \{0,2,4,6\}$. To see this, if $i\geq 8$, the $\tO$-blocks are too large in $x_3^2$ and we can check all other possible values of $i$ directly. In all theses cases the contribution is exactly of two $\tL$-bits. The summand $y_{3,2}=x_3x_2$ can only interfere with $y_{4,1}=1$ and $y_{4,0}=1$. Since $\mathsf{x}_{3,2} \in \{ \tL\tO\tL\tL\tL\tL\tO\tL, \tL\tO\tO\tO\tL\tL\tO\tL\tL\tL, \tL\tO\tO\tO\tO\tO\tO\tL\tL\tL\tL\tL,\tL\tL\tL\tL\tL\tL\tO\tL\tL\tL\tL\tL\tL\}$, all of theses contributions are more than one $\tL$-bit.

\item Suppose now that the contribution between $y_{2,0}$ and $y_{1,1}$ is exactly two $\tL$-bits. We then have $\mathsf{x}_2=\tL\tO\tL\tL\tL\tL\tL$ or $\mathsf{x}_2=\tL\tL\tL\tL\tL\tO\tL$. Moreover, $y_{4,2}=x_2$ can only interfere with $y_{3,3}=x_3^2$ and this contribution has to be exactly one $\tL$-bit. However, in both cases, this contribution exceeds one $\tL$-bit, by a similar argument as before for the different values of $i$.
\end{enumerate}

\medskip 

For the \underline{fourth line}, the summand $y_{2,0}=x_2$ can only interfere with $y_{1,1}=x_1$ with at most three $\tL$-bits. As in the precedent case, we have a contribution of one $\tL$-bit if and only if $\mathsf{x}_2={ \tL\tL\tL\tL\tL\tL\tL}$, resp., a contribution of two $\tL$-bits if and only if $$\mathsf{x}_2 \in \{\tL\tO^{a}\tL\tL\tL\tL\tL\tL:\; a\geq 1\}\;\cup\; \{\tL\tL\tL\tL\tL\tL\tO^{a}\tL:\; a\geq 1\}:=B_2,$$ resp., a contribution of three $\tL$-bits if and only if 
\begin{align*}
  \mathsf{x}_2 \in\; &\{ \tL\tO^{a'}\tL\tO^{a}\tL\tL\tL\tL\tL:\;a,a'\geq 1\}\;\cup\; \{\tL\tO^{a'}\tL\tL\tL\tL\tL\tO^{a}\tL:\;a,a'\geq 1\}\\
  &\qquad \cup \;\{ \tL\tL\tL\tL\tL\tO^{a'}\tL\tO^{a}\tL :\;a,a'\geq 1\}.
\end{align*}
Denote this last union by $B_3$.

\begin{enumerate}
\item If $\mathsf{x}_2 \in B_3$ then $y_{4,2}$ can only interfere with $y_{3,3}=1$ resulting in one $\tL$-bit. This is not possible since $y_{4,2}=x_2$. 

\item If $\mathsf{x}_2 \in B_2$ then $y_{4,2}=x_2$ can only interfere with $y_{3,3}=1$ with at most two $\tL$-bits. If this contribution does not exceed two bits, it has to be exactly two $\tL$-bits by the form of the binary expansion of $x_2$. This implies that the summand $y_{3,2}=x_2$ can interfere with $y_{4,1}=1$ and $y_{4,0}=1$ with at most one $\tL$-bit. This is not possible. 

\item If $\mathsf{x}_2={ \tL\tL\tL\tL\tL\tL\tL\tL}$ then we have $\mathsf{y}_{2,2}={ \tL\tL\tL\tL\tL\tL\tO\tO\tO\tO\tO\tO\tO\tL}$ can interfere with $y_{4,1}$, $y_{4,0}$, $y_{3,1}$ and $y_{3,0}$ and all contributions should be equal to {$ \tL$} in the end. All theses contributions have more than one $\tL$-bit since the $\tO$-block is too large. 
\end{enumerate} 
\end{enumerate} 

Therefore, the proof of \Cref{main_lemma} is complete for $m=4$. 

\subsection{The case $m=5$}
There are eleven independent summands $y_{5,5},\ldots,y_{5,0},\ldots,y_{0,0}$ each contributing to $n^2$ with exactly one $\tL$-bit.  We have not drawn the interference graph since it gets too large and it is not needed to follow the argument. Recall that we here necessarily have \boxed{k=11}. Since $y_{5,5}=y_{5,4}=y_{1,0}=y_{0,0}=1$ then $x_5=x_4=x_1=x_0=1$. The summand $y_{5,3}=x_3$ can only interfere with $y_{4,4}=1$ and the result of these two summands has to be a power of two. The same remark is also true for $y_{2,0}$ which could only interfere with $y_{1,1}=1$. We therefore have $\mathsf{x}_3=\tL^{n_3}$ and $\mathsf{x}_2 =\tL^{n_2}$ for some $n_3 \geq 1$ and $n_2\geq 1$ such that $n_3+n_2=7$. \begin{enumerate}
\item If \underline{$n_3=6$}, i.e. $\mathsf{x}_3=\tL\tL\tL\tL\tL\tL$, then $\mathsf{y}_{3,3}=\tL\tL\tL\tL\tL\tO\tO\tO\tO\tO\tO\tL$ can interfere with three of the summands $y_{5,2}, y_{5,1}, y_{5,0}, y_{4,2}, y_{4,1}$, and $y_{4,0}$. As each of these six summands is $1$, this contribution to $n^2$ has more than $\tL$-bit since the $\tO$-block of $y_{3,3}$ containing 6 consecutive $\tO$-bits is too large. 
\item If \underline{$n_3=5$}, i.e. $\mathsf{x}_3=\tL\tL\tL\tL\tL$ and $\mathsf{x}_2=\tL\tL$. Then $\mathsf{y}_{3,3}=\tL\tL\tL\tL\tO\tO\tO\tO\tO\tL$ can interfere with three of the summands $\mathsf{y}_{5,2}=\mathsf{y}_{4,2}=\tL\tL$, $y_{5,1}=y_{5,0}=y_{4,1}=y_{4,0}=1$. We conclude as done previously.
\item If \underline{$n_3=4$}, i.e $\mathsf{x}_3=\tL\tL\tL\tL$ and $\mathsf{x}_2= \tL\tL\tL$ then $\mathsf{y}_{3,3} = \tL\tL\tL\tO\tO\tO\tO\tL$ can interfere with three of the summands $\mathsf{y}_{5,2}=\mathsf{y}_{4,2}=\tL\tL\tL$, $ y_{5,1}=y_{5,0}=y_{4,1}=y_{4,0}=1$. There exists a solution to this which is graphically presented as follows: \begin{align*}
\tL\tL\tL\tL\tL\tO\tO\tO\tO&\tL \\ \tL\tL\tL\phantom{\tL}& \\ \tL& \\ &\tL
\end{align*}
Here we have to study other situations of interference in order to deduce a contradiction. The following pairs of summands have to interfere and to contribute with one $\tL$-bit: $(y_{2,0},y_{1,1})$, $(y_{3,0},y_{2,1})$, $(y_{5,2},y_{4,3})$, $(y_{5,3},y_{4,4})$. More precisely, their last significant bits have to align:
\begin{equation*}
\left\{
\begin{array}{llll}
\hat{\ell_{2}} +1=2\hat{\ell_{1}}.  \\
\hat{\ell_{3}} +1=\hat{\ell_{2}} +\hat{\ell_{1}} +1. \\
\hat{\ell_{5}} +\hat{\ell_{2}} +1=\hat{\ell_{4}} +\hat{\ell_{3}} +1. \\
\hat{\ell_{5}} +\hat{\ell_{3}} +1=2\hat{\ell_{4}}.
\end{array}
\right.
\end{equation*}
This implies $\hat{\ell_{2}}=2\hat{\ell_{1}}-1$, $\hat{\ell_{3}}=3\hat{\ell_{1}}-1$, $\hat{\ell_{4}}=4\hat{\ell_{1}}$ and $\hat{\ell_{5}}=5\hat{\ell_{1}}$. Thus the summand $y_{3,3}$ does not align in a correct way with all the other summands. In fact, since we are looking for one $\tL$-bit of contribution, we need an alignment of the first bit. Here this is not the case since all the following values are different \begin{equation*}
\left\{
\begin{array}{llll}
2\hat{\ell_{3}}=6\hat{\ell_{1}}-2 \\
\hat{\ell_{5}}+\hat{\ell_{2}}+1=7\hat{\ell_{1}}. \\
\hat{\ell_{4}}+\hat{\ell_{2}}+1=6\hat{\ell_{1}}. \\
\hat{\ell_{5}}+\hat{\ell_{1}}+1=6\hat{\ell_{1}}+1. \\
\hat{\ell_{5}}+1=5\hat{\ell_{1}}+1. \\
\hat{\ell_{4}}+\hat{\ell_{1}}+1=5\hat{\ell_{1}}+1. \\
\hat{\ell_{4}}+1=4\hat{\ell_{1}}+1. 
\end{array}
\right.
\end{equation*}
This provides us with the wanted contradiction.
\item If $1\leq n_3\leq 3$ then it turns out that $\mathsf{x}_2 =\tL^{n_2}$ for some $4 \leq n_2 \leq 6$ which is symmetric to one of the previous cases.

\end{enumerate}

Therefore, the proof of \Cref{main_lemma} is complete for $m=5$, and since all cases $1\leq m\leq 5$ are treated, this finishes the proof of  \Cref{main_lemma} and therefore of \Cref{main_theorem}.

\section{On the equation $s(n^2)\in\{4,5\}$} \label{sec-45}

The aim in this section is to study the equations $s(n^2)\in\{4,5\}$ in odd integers. Compared to the previous sections, the point of view is different here as there is no precondition on the weight of $n$. Since there is neither an \emph{a priori} bound on the length of $n$, a simple direct computation is not sufficient to determine finiteness of solutions. Our aim is to solve these equations for all $n$ composed by as many $\tL$-bits as possible. The heuristic is that the larger the weight of $n$ the more unlikely such an $n$ can be solution of $s(n^2)\leq 5$, since carry propagations have to cancel out more and more bits. In view of Conjecture~\ref{conjectureE4}, under this heuristic, this shows that it is more and more improbable to find new solutions other than those given by that conjecture.

\subsection{The case $s(n^2)=4$}\label{subn24} 

Let $n$ be an odd integer with $s(n)=k \geq 9$ such that $s(n^2)=4$. To start with, we can suppose that $k\geq 9$ since all the other cases are done in~\cite{harelaishramstoll2011}, but our algorithms could also handle smaller $k$. Write $n=1+2^{\ell}m$ with $\ell \geq 1$ and $m$ an odd integer which satisfies $s(m)=k-1$. Thus we have $n^2=1+2^{\ell+1}m+2^{2\ell}m^2$ and it implies that $s(2^{\ell+1}m+2^{2\ell}m^2)=3$. Otherwise said, we have 
\begin{equation}\label{mprod3}
    s(m+2^{\ell-1}m^2)=3.
\end{equation}

At this point there are two possible ways to attack the problem. The first one would be to use \Cref{s(ab)=3}, more precisely, a specific case in its proof where the upper bound can be improved, see~\cite{kanekostoll2021}. This allows to get
\begin{equation}\label{firstboundl}
  m(1+2^{\ell-1}m)\leq 2^{k(k-1)-13}, 
\end{equation}
and  therefore we would have $m<2^{k(k-1)-(\ell-1)/2-2}$ which in turn implies a bound on $\ell$ and what would remain is to use the algorithm \texttt{next} for each such $\ell$ to find the set of the solutions. However, this method is not sufficient for the case $s(n^2)=5$ and this is the main reason that we have created the algorithm \texttt{max-integer} that we describe shortly in the sequel (we give a more detailed description in Section~\ref{sec-algos}). 

\bigskip

According to~(\ref{mprod3}) we have to allocate three $\tL$-bits in the sum $S=m+2^{\ell-1}m^2$. In order to do so, we use the basic fact that if two integers $a,b$ satisfy $a \equiv b \pmod{2^{\lambda}}$ then $a^2 \equiv b^2 \pmod{2^{\lambda+1}}$ for all $\lambda\geq 2$, and if $a\equiv b \pmod{2}$ then $a^2 \equiv b^2 \pmod 8$. In this way if we write the binary decomposition of $m$ bit by bit from the least significant digits to the highest significant digits, then we can also deduce at the same time the binary decomposition of $2^{\ell-1}m^2$ bit by bit again from the least significant digits to the highest significant digits. The algorithm tests if the next bit in the binary decomposition of $m$ could be a $\tL$-bit or a $\tO$-bit in order to satisfy~(\ref{mprod3}). Since we have supposed a bound on the weight of $m$, the algorithm stops when the allowed amount of $\tL$-bits is reached.  

\bigskip

We show an example where we suppose $\ell=1$ and the $\ell(m)$ least significant digits of $S$ to be $\tO$. The first (rightmost) bit we add in the binary structure of $m$ is a $\tL$-bit in order to propagate the carry. We can deduce the second bit of $m^2$ but in this case it is already determined, it is a $\tO$-bit. Since we have supposed that there are no $\tL$-bits on the lower significant part of the sum $S$, the third bit of $m$ is necessarily a $\tL$-bit, and $m\equiv 7 \pmod 8$. Thus $m^2 \equiv 1 \pmod{16}$, and the fourth bit of $m^2$ is a $\tO$-bit. By iterating this argument we see that we can only add $\tL$-bits in $m$ and we obtain the following sum with a block of $(k-1)$ $\tL$-bits in $m$ (as before, we write (*) for an arbitrary finite string of bits). 

\bigskip

\begin{center}
\begin{tabular}{l l l l l l l l l l l l l l l l  l l l l l l l l l l }
   &  && &   &$\tL$ &$\cdots$ & $\tL$ & $\tL$ & $\tL$ &&&$=\;m$\\
   & +&& & (*) &$\tO$ &$\cdots$ & $\tO$ & $\tO$ & $\tL$ &&&$=\;m^2$ \\
  \hline
   & &&& (*) & $\tO$ & $\cdots$ & $\tO$ & $\tO$ & $\tO$ &&&$=\;S.$
\end{tabular}
\end{center}

\bigskip

The algorithm also considers the cases where the right part of $S$ contains one and two $\tL$-bits, see \Cref{sec-algos}, the case above  describes the main idea of the algorithm. 
\bigskip

For the search algorithm to work efficiently, we are interested in finding good bounds for $\ell$. In fact, there is a much better bound for $\ell$ than the one given by~(\ref{firstboundl}): 

\begin{lem}
Let $n$ be an odd integer such that $s(n)=k\geq 4$, $s(n^2)=4$ and $n=1+2^{\ell}m$ with $\ell\geq 2$ and $m$ an odd integer. Then we have $\ell \leq 2k$. 
\end{lem}

\begin{proof}

%
%
%

Suppose that $\ell>2k$ and set $S=m+2^{\ell-1}m^2$. Then $s(S)=3$ and $S$ is an odd integer. We consider the following addition ($\omega$, $\omega'$ are binary words, and $\varepsilon_i\in\{\tO, \tL\}$): \begin{center}
\begin{tabular}{l l l l l l l l l l l l l l l l  l l l l l l l l l l }
   &&&&&&  $\omega$ & $\varepsilon_{\ell-3}$ &$\cdots$ &$\varepsilon_{0}$ & $\tL$ &&$=\;m$\\
   & +&& & (*) & $\omega'$ &  $\tL$ & &&&&&$=\;2^{\ell-1}m^2$ \\
  \hline
   & &&& (*)& (*) & (*) & $\varepsilon_{\ell-3}$ &$\cdots$ &$\varepsilon_{0}$ &$\tL$ &&$=\;S.$
\end{tabular}
\end{center}

The block $\varepsilon_{\ell-3}\cdots \varepsilon_0$ is composed of at most one $\tL$-bit since $S$ contains three $\tL$-bits and carries propagate only to the higher significant digits.
We distinguish two cases according to the $\ell(m)$ lowest significant bits of $S$ (note that this part contains well the contribution of $\omega$ from the first summand $m$ and its interference with the rightmost $\tL$-bit of $2^{\ell-1}m^2$). This part will be called the (binary) \emph{right part of $S$}. It contains at least one $\tL$-bit (the parity bit), and at most two $\tL$-bits (which includes the parity bit). It cannot contain three $\tL$-bits since the second summand has a binary expansion strictly longer than the first summand. 

\medskip

We write $w$ for the integer whose binary expansion corresponds to $\omega$.

\medskip

\begin{enumerate}
    \item \underline{The right part of $S$ contains only the parity $\tL$-bit.} This implies that $\varepsilon_i=\tO$ for all $0\leq i\leq \ell-3$. This means that $m=1+2^{\ell-1}w$ and $m^2=1+2^{\ell}w+2^{2\ell-2}w^2$. Thus the $\ell$ lower significant bits of $m^2$ are all $\tO$-bits except the parity $\tL$-digit:

\medskip
\begin{center}
\begin{tabular}{l l l l l l l l l l l l l l l l l l  l l l l l l l l l l }
   &&&&&&&&&  $\omega$ & $\tO$ &$\cdots$ &$\tO$ & $\tL$ &&$=\;1+2^{\ell-1}w$\\
   & +&& & & $\omega$& $\tO$ &$\cdots$ &$\tO$ &  $\tL$ & &&&&&$=\;2^{\ell-1}(1+2^{\ell}w)$ \\
     & +&$\omega^2$& & $\tO$ & $\cdots$& $\tO$ &$\cdots$ &$\tO$ &  $\tO$ & &&&&&$=\;2^{\ell-1}\times 2^{2\ell-2}w^2$ \\\hline
   & &&& (*)& (*) & & (*)& & & $\tO$ &$\cdots$ &$\tO$ &$\tL$ &&$=\;S.$
\end{tabular}
\end{center}    

\medskip
Now, consider the additions of $\omega$ and $\tL$ in the middle part between the first and the second summand. Since $\ell>2k>k$, the word $\omega$ in the second summand does not interfere with the $\omega$ of the first summand.  This implies that $\omega$ is a single block of $(k-1)$ consecutive $\tL$-bits since otherwise the carry does propagate sufficiently far. This implies that
$$
m=1+2^{\ell-1}(2^{k-1}-1).
$$
Thus $m^2=1+2^{\ell}(2^{k-1}-1)+2^{2\ell-2}(2^{2k-2}-2^{k}+1)$, and we have
\begin{align*}
    m+2^{\ell-1}m^2&=1+2^{\ell-1}(2^{k-1}-1)+2^{\ell-1}+2^{2\ell-1}(2^{k-1}-1)+2^{3\ell-3}(2^{2k-2}-2^{k}+1) \\ &=1+2^{\ell-2+k}+2^{2\ell-1}(2^{k-1}-1)+2^{3\ell-3}(2^{2k-2}-2^{k}+1).
\end{align*}
Since $\ell>2k$, the terms in the above sum are non-interfering and therefore $m+2^{\ell-1}m^2$ has too many $\tL$-bits. 

\medskip

\item \underline{The right part of $S$ contains two isolated $\tL$-bits.} 

There are two cases: 

\medskip 

If this $\tL$-bit is located within the block $\varepsilon_{\ell-3}\cdots\varepsilon_0$ then there exists $i_0$ such that $\varepsilon_{i_0}=\tL$ and $\varepsilon_{i}=\tO$ for $i\neq i_0$. Thus $m=1+2^{i_0}+2^{\ell-1}w$, with $1\leq i_0 <\ell-1$, and $$2^{\ell-1}m^2=2^{\ell-1}+2^{i_0+\ell}+2^{2i_0+\ell-1}+2^{2\ell-1}w+2^{i_0+2\ell-1}w+2^{3\ell-3}w^2.$$

\begin{enumerate}
     
\item If $2i_0\geq \ell$ then we have $2^{\ell-1}m^2=2^{\ell-1}+2^{i_0+\ell}+2^{2\ell-1}w'$ for some integer $w'$. With a similar argument as in the former case, we get $$
m=1+2^{i_0}+2^{\ell-1}((2^{i_0}-1)+2^{i_0+1}(2^{k'}-1)),
$$
for some $k'\geq 0$ with $k'+i_0=k-2$. This leads to $i_0\geq \ell/2>k>k-2$, which gives a contradiction.

\item If $\ell/4<i_0<\ell/2$ then we have $2^{2i_0+\ell-1}<2^{2\ell}$ and this implies that  $m$ has the form
$$
m=1+2^{i_0}+2^{\ell-1}((2^{i_0}-1)+2^{i_0+1}(2^{i_0-1}-1)+2^{2i_0+1}(2^{k'}-1)).
$$

This leads to $s(m)\geq 2i_0> \ell/2>k$, which gives again a contradiction.

\item If $i_0\leq \ell/4$, then we have $2^{\ell-1}m^2=2^{\ell-1}+2^{i_0+\ell}+2^{2i_0+\ell-1}+2^{2\ell-1}w'$ for some integer $w'$ and we obtain
$$
m=1+2^{i_0}+2^{\ell-1}((2^{i_0}-1)+2^{i_0+1}(2^{i_0-1}-1)+2^{2i_0+1}(2^{\ell-2i_0-1}-1)).
$$

This leads to $s(m)\geq \ell-2i_0\geq \ell/2>k$, which gives a contradiction.
\end{enumerate}

\bigskip
If $\varepsilon_i=\tO$ for all $i$ then we have two remaining cases. If $w$ is even then we can use the same reasoning as before in the case (2) (a) since $S$ has two isolated $\tL$-bits. If $w$ is odd then we write $\omega=\omega'\tO\tL^{\lambda}$, with a possibly empty binary word $\omega'$ and $1\leq \lambda \leq k-1$. We can suppose that $\omega'$ is not empty since $\omega=\tL^{k-1}$ is already done in the case (1). Otherwise, $S$ has two isolated $\tL$-bits and $\omega'$ is composed by a $\tO$-block of length at least $\ell$. Therefore we can conclude with the same argument as before.
\end{enumerate}  
\end{proof}

Our implementation of the algorithm \texttt{max-integer} shows that all odd solutions $n$ such that $s(n)\leq 17$ are $(\ell,m) \in \{(3,2),(1,7),(1,23),(1,55)\}$ and this translates into 
$$\bigcup_{\lambda\leq 17 }E_{4,\lambda}=\{13,\;15,\;47,\;111\},$$
which is~(\ref{union4lambda}) in \Cref{thm_s(n2)=4 and 5}.


\subsection{The case $s(n^2)=5$}

The following result implies~(\ref{Edecomp}) in \Cref{thm_s(n2)=4 and 5}.

\begin{lem}\label{lemn25}
Let $n$ be an odd integer such that $s(n^2)=5$. Then: 
\begin{enumerate}
\item There is only a finite number of odd $n$ such that $s(n)\geq 4$.
    \item If $s(n)=3$, then $n$ is of the form $$1+2^{\ell}+2^{\ell+1}, \quad 1+2+2^{\ell},\quad\mbox{ or }\quad  1+2^{\ell}+2^{2\ell-1},$$
    for some $\ell \geq 3$.
\end{enumerate}

\end{lem}

\begin{proof}
We adapt \Cref{factorization} when the amount of $\tL$-bits in the square is fixed to be $5$. The implied constant (i.e. the constant $N_k$ appearing in its proof) will be different but we still we get that if there were an infinite number of solutions, then almost all (i.e. all but a finite number) of these solutions can be factorized this way. We again distinguish according to the number $m$ of blocks in the factorization. 

\medskip
\begin{itemize}
    \item \underline{$m=1$.} We have $(n)_2=\mathsf{x}_1\tO\cdots\tO \mathsf{x}_0$, with a large inner block of $\tO$-bits. By symmetry we can suppose that $s(x_1)\geq s(x_0)$. Since we have the three independents contributions $x_1^2$, $x_1\cdot x_0$ and $x_0^2$ for $n^2$, we see that exactly one of them has to contain one single $\tL$-bit. Thus $x_0=1$ andl $s(x_1^2)+s(x_1)=4$, i.e $x_1=3$. Then $n$ is on the form $1+2^{\ell}+2^{\ell+1}$ for sufficiently large $\ell$. We can easily check that this form is valid for all $\ell \geq 3$. By symmetry we have a second infinite family, namely $1+2+2^{\ell}$ for $\ell \geq 3$. 
    \medskip
    \item \underline{$m=2$.} We use the interference graph given in Figure~\ref{fig m=2} to deduce that $x_2=x_1=x_0=1$ and the contribution of $x_1^2+x_2\cdot x_0$ is only of one $\tL$-bit. This implies that $n$ is of the form $1+2^{\ell}+2^{2\ell-1}$ for sufficiently large $\ell$. As before, we can check that this form is valid for all $\ell \geq 3$. 
\end{itemize}
\end{proof}

 We now show how to obtain~(\ref{union5lambda}) via the algorithm \texttt{max-integer}. The method used here is similar to the previous case. We suppose $k\geq 4$ to avoid the infinite families in Lemma~\ref{lemn25}. Let $n$ be an odd integer such that $s(n)=k \geq 4$ and $s(n^2)=5$.  Let us write $n=1+2^{\ell_1}+2^{\ell_1+\ell_2}m$ with $m$ an odd integer with $s(m)=k-2$ and $\ell_1,\ell_2 \geq 1$. We have \begin{align} \label{n2=5}
n^2=1+2^{\ell_1+1}+2^{2\ell_1}+2^{\ell_1+\ell_2+1}m+2^{2\ell_1+\ell_2+1}m+2^{2\ell_1+2\ell_2}m^2.
\end{align}

We evaluate the number of isolated bits and deal with different cases according to the values of $\ell_1$ and $\ell_2$.

\begin{enumerate}
\item \underline{$\ell_1=1$.} Here~\eqref{n2=5} becomes \begin{align*}
n^2=1+2^{3}+2^{\ell_2+2}m+2^{\ell_2+3}m+2^{2\ell_2+2}m^2.
\end{align*} Two subcases cases arise: \begin{enumerate}
\item \underline{$\ell_2>1$.} We here have two isolated bits (associated with the powers $1$ and $2^3$) and this leads to \begin{align*}
s(m\cdot (3+2^{\ell_2}m))=3.
\end{align*}
By a small adaptation of the algorithm \texttt{max-integer}, we find that the only solutions are $\ell_2=2,m=11$ and $\ell_2=3,m=3$. Thus $n=51$ and $n=91$ satisfy $s(n^2)=5$. 
\item  \underline{$\ell_2=1$.} We have 
$$s(1+3m+2m^2)=s((2m+1)\cdot (m+1))=4.$$
Again, we adapt the algorithm \texttt{max-integer} and the solutions for $m$ is the set $$\{7,\; 19,\; 23,\; 55,\; 69,\; 119,\; 181,\; 367 \}.$$ The set of solutions for $n$ is therefore $$\{31,\;79,\;95,\;223,\;279,\;479,\;727,\;1471\}.$$ We mention that it is this case that motivated us to create the algorithm \texttt{max-integer} since the results from~\cite{kanekostoll2021} are not sufficient to conclude.  
\end{enumerate}
\medskip

\item \underline{$\ell_1>1$.}  This leads to two isolated bits (corresponding to the power $1$ and $2^{\ell_1+1}$). Thus~\eqref{n2=5} becomes \begin{align*}
s(2^{\ell_1}+2^{\ell_2+1}m+2^{\ell_1+\ell_2+1}m+2^{\ell_1+2\ell_2}m^2)=3.
\end{align*} 

A last adaptation of the algorithm gives $n\in\{29,\;157,\;5793\}$ as the solution set. 
\end{enumerate}

Thus we have 
$$\bigcup_{4\leq \lambda\leq 15 }E_{5,\lambda}=\{29,\;31,\;51,\;79,\;91,\;95,\;157,\;223,\;279,\;479,\;727,\;1471,\;5793\},$$
which is~(\ref{union5lambda}). Note that the solution with the largest weight is $1471$ with $s(1471)=9$. 

\section{Description of the algorithms}\label{sec-algos}

\subsection{Algorithm \texttt{next}} 

The aim is to generate efficiently all odd integers smaller than a fixed bound with a fixed weight, and most importantly, the sets
$$\Delta_{\ell_1,\ell_2,m}=\{n\in \mathbb{N}: \quad s(n)=\ell_1, \quad s(n^2)\leq \ell_2, \quad n<2^m, \quad n\text{ odd} \},$$
that we needed for our applications. 
The following result gives, starting from a given integer, the smallest integer with same weight larger than the given integer .

\begin{lem}\label{lemnext}
Let $n\geq 1$ be an integer. Write $(n)_2= \mathsf{x}\tO\tL^{b+1}\tO^{c}$ for some $b,c\geq 0$ and $\mathsf{x}$ a possibly empty binary word. Then the next integer by increasing order, denoted by $m$, with $s(m)=s(n)$ is $(m)_2=\mathsf{x}\tL\tO^{c+1}\tL^{b}$. 
\end{lem}

\begin{proof}
It is clear that $s(m)=s(n)$ and suppose there exist an integer $p$ such that $s(p)=s(m)$ and $n<p\leq m$. Since $p>n$ and $s(p)=s(n)$, a bit of index $\geq c+b+1$ of $p$ is $\tL$ because $\tL^{b+1}\tO^c$ is the expansion of the largest integer of length $c+b+1$ with weight $b+1$. Since $p\leq m$, this index is exactly $c+b+1$. By $p\leq m$, the binary expansion of $p$ begins with a $\tL$-block of length $b$. This implies $p=m$. 
\end{proof}

The algorithm \texttt{next} is a translation of Lemma~\ref{lemnext}. Given an integer $n$, the algorithm constructs the next integer by increasing order with same weight. 

\bigskip

\begin{minipage}{0.9\linewidth}
\begin{algorithm}[H]
	{\textbf{Procedure} next($n$):}
	{ \\
	$c=$ \textit{index of the least significant set bit $n$} \; 
	$n=n/2^c$ \; 
	$n=n+1$ \;
	$b=$ \textit{index of the least significant set bit $n$} -1 \;
	$n=2^c \cdot n$ \;
	$n=n | 2^b$ \Comment*[r]{| is the OR operator}
	$n=n-1$ \;
	}
\caption{\texttt{next}}
\end{algorithm}
\end{minipage}

\bigskip

Now, having constructed the set of integers with fixed weight, the second step is to determine the weight of their squares. The program uses the fact that for an integer $n$ of the form $n=m+2^{L}p$ and $m<2^{L}$, we have $n^2=m^2+2^{L+1}mp+2^{2L}p^2$. Thus the  $L+1$ lowest significant digits of $n^2$ are determined by $m$, i.e. the lower part of $n$. In our study, we are interested in integers whose squares contain only a small number of $\tL$-bits. As a consequence, if the lower part of $n^2$ contains already too many $\tL$-bits, then we can already reject the integer as a solution, and it is not necessary to compute explicitly all the square $n^2$. This preliminary calculus reduces drastically the computation time. For efficiency and practical issues, we have implemented this algorithm with $L=64$.  

We have parallelized our program and distributed the calculation on multiple threads according to a suffix before making the \texttt{next} procedure. Indeed, for a fixed $a$ we can consider integers of the form $n=1+2^{a}+2^{a+1}m$ for odd $m$ and to find \texttt{next}$(n)$ it is sufficient to execute \texttt{next}$(m)$. This is   equivalent to fix the place of the second $\tL$-bit in $n$. The cutting is therefore done via
$$\Delta_{\ell_1,\ell_2,a,m}=\{n\in \mathbb{N}: \; s(n)=\ell_1, \; s(n^2)\leq \ell_2, \; \text{ the second bit of $n$ is $a$ },\; n<2^m, \; n\text{ odd} \}$$
and
$$\Delta_{\ell_1,\ell_2,m}=\bigcup_{a=1}^{m-\ell_1+1} \Delta_{\ell_1,\ell_2,a,m}.$$ This cutting was necessary to conclude for the case $\boxed{k=11}$ of~\eqref{problem}, when $m=1$.  

\medskip
Another issue arises with this parallelization. The number of integers in $\Delta_{\ell_1,\ell_2,a,m}$ is not equivalent. The smaller the value if $a$, the larger is the cardinality of $\Delta_{\ell_1,\ell_2,a,m}$. We have supposed that the time of computation for each integer is similar (this is heuristically supported by the use of the same binomial coefficients). With a preliminary calculation, we designed specific implementations for each thread. By doing so, we could again reduce the global computation time. 

\subsection{Algorithm \texttt{max-integer}}

We describe the algorithm for the equation $s(n^2)=4$, which can be written as $s(y+2^{\ell-1}y^2)=3$ for a fixed weight $s(y)=k$ (see Section~\ref{subn24}). The other cases are similar and only need some minor changes in the implementation.

\medskip
Denote by $\lambda$ the unique integer such that $2^{\lambda-1}\leq y <2^{\lambda}$ and consider the following scheme for $y+2^{\ell-1}y^2$:

\bigskip

\begin{center}
\begin{tabular}{c c c | c  c  c c c c c c l}
   &&& $\tL$ & $\cdots$ &$\varepsilon_{\ell}$ &$\cdots$ & $\varepsilon_1$ &  $\tL$&&&$=\;y$\\
 $\tL$ & $\cdots$&$\eta_{\lambda-1}$ & $\cdots$ & $\cdots$ &  $\tL$ &&&&&& =\; $2^{\ell-1}y^2$ \\
  \hline
   & $y_1$&&&&$y_2$ && 
\end{tabular}
\end{center}

\bigskip

We cut the sum into two binary blocks, $y_1$ and $y_2$. We have to allocate in total three $\tL$-bits for $y_2$ and $y_1$. We know that $s(y_1) \geq 1$ since the most significant digit of $y+2^{\ell-1}y^2$ lies in the $y_1$-part. Let us focus on the case where $\ell=1$, the other cases are similar.

\bigskip

As explained before, we tackle this problem step by step by adding bits in the binary decomposition of $y$. In this algorithm, we consider the binary blocks such as $\tO\tL\tL$ and $\tL\tL$ to be different since we have more knowledge for the first block. In fact, in this context, it is more useful to see them as words rather than integers.

\medskip
We say that a binary word $\omega$ is a \textit{candidate} if the right part  of the sum of $\omega+\omega^2$ (by a slight abuse of the notation) has at most two $\tL$-bits for a certain length of the block. If $\omega$ is a candidate then we can extend $\omega$ to $\tO \omega$ and $\tL \omega$ to the left and check if these two new words are again candidates. If a word $\omega$ is not a candidate, then it is not possible to extend it to a candidate word since the lower bits contains already too many $\tL$-bits and these bits are not influenced by adding new bits to $\omega$ since carry propagation is directed towards the higher significant digits. The algorithm starts with the word $\omega=\tL$, constructs candidates, translates them into integers and checks whether they satisfy $s(n^2)=4$. The algorithm stops when candidates cannot be extended. 

\medskip
For the algorithm to stop, we have two conditions. The first condition is at the core of the algorithm: a word that is already of weight $k$ cannot be extended anymore with additional $\tL$-bits, so candidates have $\leq k$ $\tL$-bits. The second condition is on the length of the possible leading block of $\tO$-bits of a candidate of the form $\tO\cdots\tO \omega$. We have the following result.


\begin{lem}
Let $\omega$ be a candidate of length $\lambda$. Then the word $\tL\tO^{\lambda}\omega$ is not a candidate. 
\end{lem}  

\begin{proof}
In this case we have the following sum \begin{center}
\begin{tabular}{l l l l l l l l l l l l l l l l l l l l l l l l l }
   &&&&&&$\tL$ & $\tO\cdots \tO$ $\omega$ &&&$=\;m$\\
+&&$\tL$&$\tO\cdots \tO$&$\omega$    & $\tO\cdots \tO$ &$\tO$ & $\omega^2$   &&&$=\;m^2$ \\ \hline
 &&$\tL$ &&$\omega$ & & $\tL$ &  $\omega^2+\omega$
\end{tabular}
\end{center}
The sum contains always more than three $\tL$-bits.
\end{proof}

Thus the algorithm is the following.

\bigskip

\begin{minipage}{0.9\linewidth}
\begin{algorithm}[H]
	{\textbf{Procedure} max-integer($k$):}
	{ \\
	S=[1] \Comment*[r]{Stack of all candidates} 
    \While{S is not empty}{
    	$\omega$=S.top() \Comment*[r]{Top element of the stack} 
    	$n=1+2\omega$ \;
    	\If {$s(n^2)=4$}{
    	print(n)}{}
    	S.pop() \Comment*[r]{remove $\omega$ from the stack}
    	\If{ $\tL\omega$ is a candidate and $s(\omega)<k$}{
    	S.push($\tL\omega)$ \Comment*[r]{add $\tL\omega$ to the top of the stack} 
    	}{}
    	\If{ $\tO\omega$ is a candidate and 2*(length of leading zeros of $\tO\omega$) $<$ length of $\omega$}{
    	S.push($\tO\omega$) \Comment*[r]{add $\tO\omega$ to the top of the stack} 
    	}{}
	}{
	}
	}
\caption{\texttt{max-integer}}
\end{algorithm}
\end{minipage}

\bigskip

With respect to Theorem~\ref{thm_s(n2)=4 and 5},~(\ref{union4lambda}), our implementation of the algorithm takes 102 sec to end for $s(n)=16$ and 2h 50min for $s(n)=17$ with a desk machine Intel(R) Core(TM) i9-9980HK CPU @ 2.40GHz. The code program is available here: \newline \url{https://gitlab.inria.fr/jamet/on-the-binary-digits-of-n-and-n2}. 

\section{The remaining cases $k=14,15$}\label{sec-rem1415}
We here consider the problem of determining the solutions of
$$s(n)=s(n^2)\in\{14, 15\},$$
which are the last two remaining cases in the original problem. These cases are much more difficult than the previous ones since we cannot rely on the former cases to resolve the problem. As already mentioned in Section~\ref{secintro},~(\ref{s(n)=s(n2)=12})--(\ref{s(n)=s(n2)=16}), there are infinitely many solutions for $s(n)=s(n^2)\in\{12,13\}$.   

\bigskip
To tackle these remaining cases, we improved our programs that determine the sets $\Delta_{\ell_1,\ell_2,m}$. For the case $k=14$, there are many more subcases than before, and the investigation gets extremely cumbersome. We still can rely on \Cref{factorization} which gives decompositions into $m$ blocks with $1\leq m \leq 6$ for sufficiently large solutions.

For the case $m=1$, we give in the sequel the constants implied by \Cref{s(ab)=2} and \Cref{s(ab)=3}. This demonstrates the difficulty of the computation. Such as before,  write $(n)_2=\mathsf{x}_1 \tO\ldots\tO \mathsf{x}_0$. By symmetry we can suppose that $s(x_1)\geq s(x_0)$. In this case we have $s(x_1)^2,s(x_0^2)\leq 10$ and the following table (see also Section~\ref{subsecm1}):

\bigskip

\begin{center}
\begin{tabular}{|c|c|c|c|c|}
\hline $s(x_1)$ & $s(x_0)$ & Sets $\Delta$ for $ s(x_1x_0)=2$ & Sets $\Delta$ for $s(x_1x_0)=3$   \\ \hline\hline
7 & 7  & $\Delta_{7,10,93}$ $\times$ $\Delta_{7,10,93}$ & $\Delta_{7,9,182}$ $\times$ $\Delta_{7,9,182}$ \\ \hline 
8 & 6 & $\Delta_{8,10,91}$ $\times$ $\Delta_{6,10,91}$ & $\Delta_{8,9,178}$ $\times$ $\Delta_{6,9,178}$\\ \hline 
9 & 5  & $\Delta_{9,10,85}$ $\times$ $\Delta_{5,10,85}$ & $\Delta_{9,9,166}$ $\times$ $\Delta_{5,9,166}$\\ \hline 
10 & 4 &  $\Delta_{10,10,75}$ $\times$ $\Delta_{4,10,75}$ & $\Delta_{10,9,146}$ $\times$ $\Delta_{4,9,146}$ \\ \hline 
11 & 3 & $\Delta_{11,10,61}$ $\times$ $\Delta_{3,10,61}$ & $\Delta_{11,9,118}$ $\times$ $\Delta_{3,9,118}$ \\ \hline 
12 & 2 & $\Delta_{12,10,43}$ $\times$ $\Delta_{2,10,43}$ & $\Delta_{12,9,82}$ $\times$ $\Delta_{2,9,82}$\\ \hline 
\end{tabular}
\end{center}

\bigskip

The algorithm \texttt{next} gives us no solution for $s(x_1\cdot x_0)=2$ and only the couple $(3695,143)$ for $s(x_1\cdot x_0)=3$. But we have $s(3695^2)+s(143^2)=17>11$ and then this couple is not a solution. Therefore there is no infinite family such that $m=1$ such as in  
(\ref{s(n)=s(n2)=12})--(\ref{s(n)=s(n2)=16}).
\bigskip

For $k=15$, we have the following sets to determine: 

\bigskip

\begin{center}
\begin{tabular}{|c|c|c|c|c|}
\hline $s(x_1)$ & $s(x_0)$ & Sets $\Delta$ for $ s(x_1x_0)=2$ & Sets $\Delta$ for $s(x_1x_0)=3$   \\ \hline\hline
8 & 7  & $\Delta_{8,11,107}$ $\times$ $\Delta_{7,11,107}$ & $\Delta_{8,10,210}$ $\times$ $\Delta_{7,10,210}$ \\ \hline 
9 & 6 & $\Delta_{9,11,103}$ $\times$ $\Delta_{6,11,103}$ & $\Delta_{9,10,202}$ $\times$ $\Delta_{6,10,202}$ \\ \hline 
10 & 5  & $\Delta_{10,11,95}$ $\times$ $\Delta_{5,11,95}$ & $\Delta_{10,10,186}$ $\times$ $\Delta_{5,10,186}$\\ \hline 
11 & 4 &  $\Delta_{11,11,83}$ $\times$ $\Delta_{4,11,83}$ & $\Delta_{11,10,163}$ $\times$ $\Delta_{4,10,163}$ \\ \hline 
12 & 3 & $\Delta_{12,11,67}$ $\times$ $\Delta_{3,11,67}$ & $\Delta_{12,10,130}$ $\times$ $\Delta_{3,10,130}$\\ \hline 
13 & 2 & $\Delta_{13,11,47}$ $\times$ $\Delta_{2,11,47}$ & $\Delta_{13,10,90}$ $\times$ $\Delta_{2,10,90}$\\ \hline 
\end{tabular}
\end{center}

\bigskip

Finally, we have made a global search for $s(n^2)=s(n)=11,12,13,14,15$ for $n<2^{80}$ with a triple cutting. We found the following proportions:

\bigskip

$$\begin{tabular}{l|p{6cm}|l|c} \
     $k$& Proportion of odd integers such that $s(n)=s(n^2)=k$ for $n<2^{80}$  & Running time & Number of cores used\\ \hline \hline
     $11$& $4\cdot 10^{-10}$ & 2min 19sec & 659 \\
     $12$ & $1.5\cdot 10^{-10}$ & 4min 58sec & 480 \\
     $13$ & $7.2\cdot 10^{-11}$ & 32min 25sec & 360 \\
     $14$ & $2.6\cdot 10^{-11}$  & 3h 34min 43sec & 277 \\
     $15$ & $1.2 \cdot 10^{-11}$ & 23h 24min 47sec & 218
\end{tabular}$$

\bigskip

These five proportions are similar but there is a clear difference between the cases $k=11$ and $k=12,13$. For $k=11$ our algorithm finds that the largest solution is $n=35463511416833$ of binary length $46$. The structure in the solutions for $k=12$ and $k=13$ is clearly different since we can see a threshold between sporadic solutions and infinite families that are composed by small blocks. These infinite families already appear before this threshold. For example, for $k=12$, we have that any solution of binary length larger than $55$ is of the form $111\cdot 2^t+111$, but this form is already valid and appears for $t\geq 15$. 

\medskip

For $k=14$, we still find solutions of binary length $80$, such as  $n=605643510452789079965697$ for example. Nevertheless, no infinite family occurs clearly. For $k=15$, the situation is similar: we have a solution of length $80$, for example, $n=605642350760526229274625$, again there is no obvious infinite family and the $\tL$-bits in the solutions do not follow an apparent rule. We believe that if an infinite family exists for $k=14$ or $k=15$, it should appear clearly for $n<2^{80}$ already, as it is the case for $k=12,13$ and $16$.
It is therefore likely that there is only a finite number of solution. We formulated this in \Cref{conj 14-15}. 

\medskip

This dichotomy between finite and infinite number of solutions for the problem~(\ref{problem}) is rather surprising but seems at the same time to occur frequently in this context, for example, such as between $E_3$ (see~(\ref{s(n2)=3})) and the conjectured set $E_4$ (see Conjecture~\ref{conjE4}). Interestingly enough, there exist again infinite (independent) families for the twisted system $$s(n)=14,\quad s(n^2)=15,$$
namely
$$n=23\cdot 2^{t}+2943, \mbox{ with }t\geq 13,$$
and 
$$n=727\cdot 2^{t}+727 \mbox{ with }t\geq 21.$$
    
\bigskip

To perform all calculations in our article, we used the cluster gros that consists of 123 nodes, Intel Xeon Gold 5220 and 18 cores / CPU, with 96 GiB of memory, see \newline  \url{https://www.grid5000.fr/w/Nancy:Hardware#gros} and the code program is available \newline \url{https://gitlab.inria.fr/jamet/on-the-binary-digits-of-n-and-n2}.

\section*{Acknowledgements} 

The authors would like to thank Lukas Spiegelhofer for discussions and a very useful \texttt{C}-program. This work was supported partly by the French PIA project “Lorraine Université d’Excellence”, reference ANR-15-IDEX-04-LUE, and by the projects ANR-18-CE40-0018 (EST) and ANR-20-CE91-0006 (ArithRand). The third author was supported by JSPS KAKENHI Grant Number 19K03439. 

\bibliographystyle{amsplain}
\bibliography{biblio}

\end{document}